\documentclass[a4paper, english]{scrartcl}

\usepackage[utf8]{inputenc}
\usepackage[T1]{fontenc}
\usepackage{lmodern}
\usepackage{amsmath, amsthm, amssymb, mathtools}
\usepackage{float}
\usepackage[dvipsnames,svgnames]{xcolor} 
\usepackage{tikz,pgf}
\usetikzlibrary{calc}
\usetikzlibrary{fadings}
\usepackage[colorlinks=true,linkcolor=RoyalBlue,citecolor=PineGreen,urlcolor=RoyalBlue]{hyperref}
\usepackage{breqn}
\usepackage{booktabs}  
\usepackage{multirow}
\usepackage{array}
\usepackage{algorithm,algorithmicx,algpseudocode}

\algnewcommand\And{\textbf{and} }
\newcommand{\consistencycheck}{\hyperref[alg:checkConsistent]{\textsc{CheckConsistency}}}
\newcommand{\consistenttypes}{\hyperref[alg:findConsistent]{\textsc{ConsistentTypes}}}

\newtheorem{theorem}{Theorem}[section]

\newtheorem{proposition}[theorem]{Proposition}
\newtheorem{corollary}[theorem]{Corollary}
\newtheorem{remark}[theorem]{Remark}

\theoremstyle{definition}
\newtheorem{definition}[theorem]{Definition}
\newtheorem{example}[theorem]{Example}

\colorlet{col1}{LimeGreen}
\colorlet{col2}{Orchid}
\colorlet{col3}{Orange}
\colorlet{col4}{Cerulean}
\colorlet{col5}{Goldenrod}
\colorlet{col6}{LightGray}
\colorlet{colGray}{LightGray}
\definecolor{colR}{rgb}{.932,.172,.172} 
\definecolor{colB}{rgb}{.255,.41,.884} 

\tikzstyle{vertex}=[circle, draw, fill=black, inner sep=0pt, minimum size=4pt]
\tikzstyle{smallvertex}=[circle, draw, fill=black, inner sep=0pt, minimum size=2pt]
\tikzstyle{edge}=[line width=1.5pt,black!50!white]
\tikzstyle{gridl}=[black!50!white]
\tikzstyle{axes}=[gridl,-latex]

\tikzstyle{midvertex}=[circle, draw, fill=black, inner sep=0pt, minimum size=3pt]

\tikzstyle{lnode}=[circle,white,draw=black!60!white,fill=black!60!white,inner sep=1pt, font=\scriptsize]
\tikzstyle{lnodesmall}=[circle,white,draw=black!60!white,fill=black!60!white,inner sep=1pt, font=\scriptsize]

\colorlet{colvR}{black!70!white}
\tikzstyle{lnodeR}=[circle,colvR,draw=colvR,fill=white,inner sep=1.2pt, font=\scriptsize]
\tikzstyle{vertexR}=[circle,thick,draw=colvR,fill=white,inner sep=0pt, minimum size=4.5pt]
\tikzstyle{midvertexR}=[circle,thick,draw=colvR, fill=white, inner sep=0pt, minimum size=3.25pt]
\tikzstyle{smallvertexR}=[circle,thick,draw=colvR, fill=white, inner sep=0pt, minimum size=2pt]

\tikzstyle{redge}=[edge,colR]
\tikzstyle{bedge}=[edge,colB]

\tikzstyle{arbitrarygraph}=[dashed,black!70!white,thick]
\tikzstyle{labelsty}=[font=\scriptsize]
\tikzstyle{indicatededge}=[pin={[pin distance=6pt,pin edge={thin,path fading=#1,colR}]5:},pin={[pin distance=6pt,pin edge={thin,path fading=#1,colR}]-5:},pin={[pin distance=6pt,pin edge={thin,path fading=#1,colR}]2:}]

\newcolumntype{C}{>{$}c<{$}} 
\newcommand{\sage}{\textsc{SageMath}}
\newcommand{\flexrilog}{\textsc{FlexRiLoG}}

\newcommand{\blue}{\text{blue}}
\newcommand{\red}{\text{red}}

\newcommand{\norm}[1]{\left\lVert#1\right\rVert}

\DeclareMathOperator{\NAC}{NAC}
\newcommand{\nac}[1]{\NAC_{#1}}
\newcommand{\nacC}[2]{\NAC_{#1}(#2)}

\newcommand{\NN}{\mathbb{N}}
\newcommand{\RR}{\mathbb{R}}
\newcommand{\CC}{\mathbb{C}}
\newcommand{\QQ}{\mathbb{Q}}
\newcommand{\ZZ}{\mathbb{Z}}
\newcommand{\ci}{i}

\newcommand{\hot}{\text{h.o.t.}}

\newcommand{\C}{\mathcal C}
\newcommand{\conj}[1]{\overline{#1}}
\newcommand{\Ktt}{{K_{3,3}}}

\DeclareMathOperator{\argmin}{arg\,min}
\DeclareMathOperator{\argmax}{arg\,max}
\DeclareMathOperator{\ord}{ord}
\DeclareMathOperator{\gap}{gap}
\DeclareMathOperator{\Ram}{ram}
\newcommand{\ram}{\Ram_{F(\C)/F(\C')}}

\DeclareMathOperator{\val}{Val}
\newcommand{\Val}[1]{\val({#1})}

\newcommand{\casesQ}[1]{\mathfrak{#1}}
\newcommand{\caseG}{\casesQ{g}}
\newcommand{\caseO}{\casesQ{o}}
\newcommand{\caseE}{\casesQ{e}}
\newcommand{\caseP}{\casesQ{p}}
\newcommand{\caseA}{\casesQ{a}}

\newcounter{casesQone}
\renewcommand{\thecasesQone}{(\alph{casesQone})}
\newcommand{\newCaseQone}{\refstepcounter{casesQone}\thecasesQone}
\newcommand{\IIplus}{\text{II}_+}
\newcommand{\IIminus}{\text{II}_-}
\newcommand{\IVplus}{\text{IV}_{\!+}}
\newcommand{\IVminus}{\text{IV}_{\!-}}

\newcommand\myeq[1]{\stackrel{\mathclap{\text{\footnotesize\mbox{#1}}}}{=}}

\newcommand{\colO}{\begin{tikzpicture}[baseline=($(a.base)!.16!(d.base)$), scale=0.4] 
			\node[smallvertex] (a) at (0,0) {};
			\node[smallvertex] (b) at (1,0) {};
			\node[smallvertex] (c) at (1,1) {};
			\node[smallvertex] (d) at (0,1) {};
			
			\draw[bedge] (a) to (b);
			\draw[redge] (b) to (c);
			\draw[bedge] (c) to (d);
			\draw[redge] (d) to (a);
		\end{tikzpicture}}
	
\newcommand{\colL}{\begin{tikzpicture}[baseline=($(a.base)!.16!(d.base)$), scale=0.4]
		\node[smallvertex] (a) at (0,0) {};
		\node[smallvertex] (b) at (1,0) {};
		\node[smallvertex] (c) at (1,1) {};
		\node[smallvertex] (d) at (0,1) {};
		
		\draw[bedge] (a) to (b);
		\draw[redge] (b) to (c);
		\draw[redge] (c) to (d);
		\draw[bedge] (d) to (a);
	\end{tikzpicture}}

\newcommand{\colR}{\begin{tikzpicture}[baseline=($(a.base)!.16!(d.base)$), scale=0.4]
		\node[smallvertex] (a) at (0,0) {};
		\node[smallvertex] (b) at (1,0) {};
		\node[smallvertex] (c) at (1,1) {};
		\node[smallvertex] (d) at (0,1) {};
		
		\draw[bedge] (a) to (b);
		\draw[bedge] (b) to (c);
		\draw[redge] (c) to (d);
		\draw[redge] (d) to (a);
	\end{tikzpicture}}

\newcommand{\colOR}{\begin{tikzpicture}[baseline=($(a.base)!.16!(d.base)$), scale=0.4] 
			\node[smallvertex] (a) at (0,0) {};
			\node[smallvertex] (b) at (1,0) {};
			\node[smallvertex] (d) at (0,1) {};
			
			\draw[bedge] (a) to (b);
			\draw[redge] (d) to (a);
		\end{tikzpicture}}
		
\newcommand{\colOL}{\begin{tikzpicture}[baseline=($(a.base)!.16!(c.base)$), scale=0.4] 
			\node[smallvertex] (a) at (0,0) {};
			\node[smallvertex] (b) at (1,0) {};
			\node[smallvertex] (c) at (1,1) {};
			
			\draw[bedge] (a) to (b);
			\draw[redge] (b) to (c);
		\end{tikzpicture}}

\title{On the Classification of Motions of Paradoxically Movable Graphs\footnote{Supported by the Austrian Science Fund (FWF): P31061, P31888 and W1214-N15.
	This project has received funding from the European Union’s Horizon 2020 research and innovation programme under the Marie Skłodowska-Curie grant agreement No 675789.}}
\author{%
Georg Grasegger\thanks{Johann Radon Institute for Computational and Applied Mathematics (RICAM), Austrian Academy of Sciences}
\and 
Jan Legersk\'y\thanks{Johannes Kepler University Linz, Research Institute for Symbolic Computation (RISC)} \thanks{Department of Applied Mathematics, Faculty of Information Technology, Czech Technical University in Prague}
\and 
Josef Schicho\footnotemark[3]
}

\date{}

\begin{document}

\maketitle
\begin{abstract}
  Edge lengths of a graph are called flexible if there exist infinitely 
	many non-congruent realizations of the graph in the plane
	satisfying these edge lengths.
	It has been shown recently that a graph has flexible edge lengths
	if and only if the graph has a special type of edge coloring called 
	NAC-coloring.
	We address the question how to determine all possible proper flexible
	edge lengths from the set of all NAC-colorings of a graph.
	We do so using restrictions to 4-cycle subgraphs.
\end{abstract}

Rigidity theory considers graphs with given labelings of edges by positive real numbers.
The number of realizations of a graph in $\RR^2$ 
such that the distances of adjacent vertices are equal to the labeling of the edges is widely studied.
Such a labeling is called flexible if the number of realizations, counted modulo rigid transformations, is infinite.
Otherwise, the labeling is called rigid.
A graph is called generically rigid if every labeling induced by a generic realization is rigid. However, it might have non-generic flexible labelings (see Figure~\ref{fig:threeprism}).
We call a graph movable if there is a proper flexible labeling, i.e., with infinitely many injective realizations, modulo rigid transformations.
In other words, we disallow realizations where two vertices coincide, 
but intersecting or partially overlapping edges are allowed (see Figure~\ref{fig:noninjective} for a non-injective example).
\begin{figure}[ht]
  \centering
  \begin{tikzpicture}[scale=1.5]
    \node[vertexR] (a) at (0,0) {};
		\node[vertexR] (b) at (1,0) {};
		\node[vertexR] (c) at (0.5,0.5) {};
		\node[vertexR] (d) at (0,1.5) {};
		\node[vertexR] (e) at (1,1.5) {};
		\node[vertexR] (f) at (0.5,1) {};
		\draw[edge] (a)edge(b) (b)edge(c) (c)edge(a) (d)edge(e) (e)edge(f) (f)edge(d) (a)edge(d) (b)edge(e) (c)edge(f);
  \end{tikzpicture}
  \qquad
  \begin{tikzpicture}[scale=1.5]
    \node[vertexR] (a) at (0,0) {};
		\node[vertexR] (b) at (1,0) {};
		\node[vertexR] (c) at (0.5,0.5) {};
		\node[vertexR] (d) at (0,1) {};
		\node[vertexR] (e) at (1,1) {};
		\node[vertexR] (f) at (0.5,1.5) {};
		\node[vertexR,rotate around=-25:(a)] (ds) at (d) {};
		\node[vertexR,rotate around=-25:(b)] (es) at (e) {};
		\node[vertexR,rotate around=-25:(c)] (fs) at (f) {};
		\draw[edge] (a)edge(b) (b)edge(c) (c)edge(a) (d)edge(e) (e)edge(f) (f)edge(d) (a)edge(d) (b)edge(e) (c)edge(f);
		\draw[edge,dashed] (ds)edge(es) (es)edge(fs) (fs)edge(ds) (a)edge(ds) (b)edge(es) (c)edge(fs);
  \end{tikzpicture}
  \caption{The three-prism graph is generically rigid (rigid labeling on the left) but has a proper flexible labeling (right).}
  \label{fig:threeprism}
\end{figure}
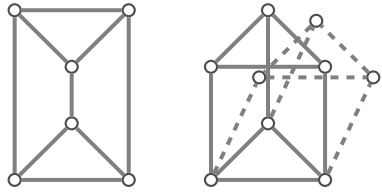

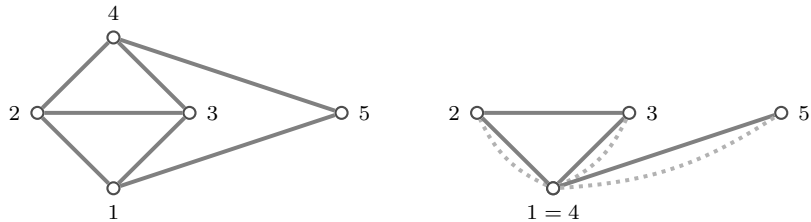
\begin{figure}[ht]
  \centering
  \begin{tikzpicture}[baseline={(0,1)}]
		\node[vertexR,label={[labelsty]-90:1}] (a) at (0,0) {};
		\node[vertexR,label={[labelsty]180:2}] (b) at (-1,1) {};
		\node[vertexR,label={[labelsty]0:3}] (c) at (1,1) {};
		\node[vertexR,label={[labelsty]0:5}] (e) at (3,1) {};
		\node[vertexR,label={[labelsty]:4}] (d) at (0,2) {};
		\draw[edge] (a) to (b);
		\draw[edge] (b) to (c);
		\draw[edge] (c) to (d);
		\draw[edge] (c) to (a);
		\draw[edge] (d) to (b);
		\draw[edge] (e) to (d);
		\draw[edge] (a) to (e);
  \end{tikzpicture}%
  \qquad
  \begin{tikzpicture}[baseline={(0,1)}]
    \node[vertexR] (a) at (0,0) {};
		\node[vertexR,label={[labelsty]180:2}] (b) at (-1,1) {};
		\node[vertexR,label={[labelsty]0:3}] (c) at (1,1) {};
		\node[vertexR,label={[labelsty]0:5}] (e) at (3,1) {};
		\node[vertexR,label={[labelsty]-90:$1=4$}] (d) at (a) {};
		\draw[edge, dotted, black!30!white] (a) to[bend right=-20] (b);
		\draw[edge] (b) to (c);
		\draw[edge] (c) to (d);
		\draw[edge, dotted, black!30!white] (c) to[bend right=-20] (a);
		\draw[edge] (d) to (b);
		\draw[edge] (e) to (d);
		\draw[edge, dotted, black!30!white] (a) to[bend right=15] (e);
  \end{tikzpicture}%
  \caption{A generically rigid graph with a flexible labeling. 
  The left realization cannot be continuously deformed but there are infinitely many 
  non-congruent realizations inducing the same labeling where two vertices overlap (right).
  Dotted edges symbolize that here edges overlap. For this graph, the labeling is flexible 
  if and only if there is a realization where the vertices 1 and 4 coincide.}
  \label{fig:noninjective}
\end{figure}

It is known from Pollaczek-Geiringer~\cite{Geiringer1927} and Laman~\cite{Laman1970},
that a graph $G=(V_G,E_G)$ is generically minimally rigid (Laman graph) if and only if $|E_G| = 2|V_G|-3$, and $|E_H|\leq 2|V_H|-3$
for all subgraphs $H$ of $G$ on at least two vertices.
Hence, graphs that do not have a spanning Laman subgraph are movable,
since a generic realization induces a proper flexible labeling.

The study of movable generically rigid graphs has a long history.
Dixon found two types of flexible labelings of the bipartite graph $K_{3,3}$ \cite{Dixon, WunderlichNineBar, Stachel}.
Walter and Husty~\cite{WalterHusty} proved that these labelings give indeed all proper flexible ones for $K_{3,3}$.
Burmester's focal point mechanism~\cite{Burmester1893},  
a graph with 9 vertices and 16 edges,
a 12-vertex graph studied by Kempe~\cite{Kempe1877},
and two constructions by Wunderlich~\cite{Wunderlich1954,Wunderlich1981} are further examples of movable generically rigid graphs.

The main question in this paper is the following: Find all proper flexible labelings of a given graph!
In \cite{flexibleLabelings}, we already provide a combinatorial characterization of graphs
with a flexible labeling: there is a flexible labeling if and only if the graph has a so-called NAC-coloring.
A NAC-coloring is a coloring of edges by two colors such that in every cycle,
either all edges have the same color or there are at least two edges of each color.
A drawback of the provided construction of a flexible labeling from a NAC-coloring
is that it does not give all possible ones.
Moreover, many Laman graphs have a NAC-coloring for which the constructed flexible labeling is not proper.
In \cite{movableGraphs} movable graphs are studied and methods for checking movability are presented.

In this paper we present methods giving necessary conditions on proper flexible labelings for a given graph.
This yields a full classification of all proper flexible labelings in some cases.
Animations with the movable graphs can be found in~\cite{LegerskyAnimations}.
The implementation of the concepts we introduce is part of the \sage{} package \flexrilog{}~\cite{flexrilog}.

The structure of the paper is the following. In Section~\ref{sec:preliminaries},
we recall NAC-colorings and some previous results.
In Section~\ref{sec:comparingLC} we derive algebraic relations for the edge lengths given some NAC-colorings
which are relevant for a motion.
Whether a NAC-coloring is possibly relevant is checked in Section~\ref{sec:ramification} using restrictions to 4-cycles.
Relevant NAC-colorings on all 4-cycle subgraphs yield a method for finding consistent NAC-colorings for the motion of 
the whole graph (Section~\ref{sec:consistency}).
Relevant NAC-colorings for motions of the 4-cycle graph with various edge lengths
 are investigated in Appendix~\ref{sec:NACsQuadrilateral}.
As a main example, in Section~\ref{sec:Q1classification}, we classify the flexible labelings of a graph with 7 vertices.

\section{Preliminaries}
\label{sec:preliminaries}

In this paper a graph $G=(V_G,E_G)$ is always connected and contains at least one edge.

\begin{definition}
	Let~$G$ be a graph and let $\lambda\colon E_G\rightarrow \RR_+$ be an edge labeling of~$G$.
	A map $\rho\colon V_G\rightarrow \RR^2$ is a \emph{realization of~$G$ compatible
	with~$\lambda$} if $\norm{\rho(u)-\rho(v)}=\lambda(uv)$ for all edges~$uv\in E_G$.
	We say that two realizations~$\rho_1$ and~$\rho_2$ are congruent
	if there exists a direct Euclidean isometry~$\sigma$ of~$\RR^2$ such that $\rho_1=\sigma \circ\rho_2$.
	The labeling $\lambda$ is called \emph{(proper) flexible}
	if the number of (injective) realizations of~$G$ compatible with~$\lambda$ up to congruence is infinite.
	We say that a graph is \emph{movable} if it has a proper flexible labeling.
\end{definition}
We are interested in generically rigid graphs with some paradoxical cases of flexible labelings. 
Given $\bar{u}\bar{v}\!\in\! E_G$, edge lengths $\lambda_{uv}=\lambda(uv)$ and unknown coordinates $(x_u,y_u)$ for $u\!\in\! V_G$
yield
\begin{align} \label{eq:mainSystemOfEquations}
	x_{\bar{u}}=0\,, \quad	y_{\bar{u}}&=0\,, \nonumber \\
	x_{\bar{v}}=\lambda_{\bar{u}\bar{v}}\,, \quad y_{\bar{v}}&=0\,, \\
	(x_u-x_v)^2+(y_u-y_v)^2&= \lambda_{uv}^2 \quad \text{ for all } uv \in E_G.	\nonumber
\end{align}
Note that we fix two adjacent vertices $\bar u,\bar v$ in order to get rid of translations and rotations.
Then the labeling $\lambda$ is flexible if and only if there are infinitely many solutions of~\eqref{eq:mainSystemOfEquations}.
It is proper flexible if infinitely many solutions satisfy	
$(x_u,y_u)\neq(x_v,y_v)$ for all distinct $u,v \in V_G$.
A graph can be generically rigid but still have a proper flexible labeling as Figure~\ref{fig:threeprism} shows.

In \cite{movableGraphs} we constructed all movable graphs up to $8$ vertices.

\begin{theorem}[\cite{movableGraphs}]\label{thm:listEightVertices}
	Let $G$ be a graph with at most $8$ vertices such that it has a spanning Laman subgraph
	and has no vertex of degree two.
	The graph $G$ is movable if and only if it is $K_{3,3}, K_{3,4}, K_{3,5}$, $K_{4,4}$,
	one of the graphs in Figure~\ref{fig:constDistClosures}, or a spanning subgraph thereof.
\end{theorem}

\begin{figure}[htb]
	\centering
		\begin{tabular}{cccccc}
			\begin{tikzpicture}[scale=1.4]
			    \node[vertex] (0) at (0.60, 0.40) {};
			    \node[vertex] (1) at (1, 0) {};
			    \node[vertex] (2) at (1, 1) {};
			    \node[vertex] (3) at (0.60, 1.4) {};
			    \node[vertex] (6) at (0, 1) {};
			    \node[vertex] (7) at (0, 0) {};
			    \draw[edge] (0)edge(1) (0)edge(7) (1)edge(7) (2)edge(3) (2)edge(6) (3)edge(6)  ;
			    \draw[edge] (0)edge(3) (1)edge(2)  (6)edge(7)  ;
			\end{tikzpicture}
			&
			\begin{tikzpicture}[scale=1.4]
			    \node[vertex] (0) at (0.60, 0.40) {};
			    \node[vertex] (1) at (1, 0) {};
			    \node[vertex] (2) at (1, 1) {};
			    \node[vertex] (3) at (0.60, 1.4) {};
			    \node[vertex] (4) at (0.50, 0.2) {};
			    \node[vertex] (6) at (0, 1) {};
			    \node[vertex] (7) at (0, 0) {};
			    \draw[edge] (0)edge(1) (0)edge(7) (1)edge(7) (2)edge(3) (2)edge(6) (3)edge(6) (4)edge(0) (4)edge(1) (4)edge(7)  ;
			    \draw[edge] (0)edge(3) (1)edge(2)  (6)edge(7)  ;
			\end{tikzpicture}
	    &
			\begin{tikzpicture}[scale=1.4]
			    \node[vertex] (0) at (0.70, 0.40) {};
			    \node[vertex] (1) at (1, 1) {};
			    \node[vertex] (2) at (0.30, 1.4) {};
			    \node[vertex] (3) at (0, 1) {};
			    \node[vertex] (4) at (0, 0) {};
			    \node[vertex] (5) at (0.30, 0.40) {};
			    \node[vertex] (6) at (1, 0) {};
			    \node[vertex] (7) at (0.70, 1.4) {};
			    \draw[edge] (0)edge(4) (0)edge(5) (0)edge(6) (1)edge(2) (1)edge(3) (1)edge(7) (2)edge(3) (2)edge(7) (3)edge(7) (4)edge(5) (4)edge(6) (5)edge(6)  ;
			    \draw[edge] (0)edge(7) (1)edge(6) (2)edge(5) (3)edge(4)  ;
			\end{tikzpicture}
			&
			\begin{tikzpicture}[scale=1.4]
			    \node[vertex] (0) at (1, 0) {};
			    \node[vertex] (1) at (-0.30, 0.30) {};
			    \node[vertex] (2) at (0.30, 0.70) {};
			    \node[vertex] (3) at (1, 1) {};
			    \node[vertex] (4) at (0.60, 1.4) {};
			    \node[vertex] (5) at (0, 1) {};
			    \node[vertex] (6) at (0.60, 0.40) {};
			    \node[vertex] (7) at (0, 0) {};
			    \draw[edge] (0)edge(6) (0)edge(7) (1)edge(2) (3)edge(4) (3)edge(5) (4)edge(5) (6)edge(7)  ;
			    \draw[edge] (0)edge(3) (1)edge(5) (1)edge(7) (2)edge(4) (2)edge(6) (4)edge(6) (5)edge(7)  ;
			\end{tikzpicture}
			&
			\begin{tikzpicture}[scale=1.4]
			    \node[vertex] (0) at (0.60, 0.40) {};
			    \node[vertex] (1) at (1, 0) {};
			    \node[vertex] (2) at (1, 1) {};
			    \node[vertex] (3) at (0.60, 1.4) {};
			    \node[vertex] (4) at (0.30, 0.50) {};
			    \node[vertex] (5) at (-0.30, 0.50) {};
			    \node[vertex] (6) at (0, 1) {};
			    \node[vertex] (7) at (0, 0) {};
			    \draw[edge] (0)edge(1) (0)edge(7) (1)edge(7) (2)edge(3) (2)edge(6) (3)edge(6)  ;
			    \draw[edge] (0)edge(3) (1)edge(2) (4)edge(5) (4)edge(6) (4)edge(7) (5)edge(6) (5)edge(7) (6)edge(7)  ;
			\end{tikzpicture}
			&
			\begin{tikzpicture}[scale=1.4]
			    \node[vertex] (0) at (0.60, 0.40) {};
			    \node[vertex] (1) at (0, 0) {};
			    \node[vertex] (2) at (1, 0) {};
			    \node[vertex] (3) at (0.80, 0.60) {};
			    \node[vertex] (4) at (0.30, 0.60) {};
			    \node[vertex] (5) at (1, 1) {};
			    \node[vertex] (6) at (0, 1) {};
			    \node[vertex] (7) at (0.60, 1.4) {};
			    \draw[edge] (0)edge(1) (0)edge(2) (1)edge(2) (3)edge(4) (3)edge(5) (3)edge(6) (3)edge(7) (4)edge(5) (4)edge(6) (4)edge(7) (5)edge(6) (5)edge(7) (6)edge(7)  ;
			    \draw[edge] (0)edge(7) (1)edge(6) (2)edge(5)  ;
			\end{tikzpicture}
			\\
			$L_1$ & $L_2$ &$L_3$ &$L_4$ &$L_5$ &$L_6$ 
		\end{tabular}
	\begin{tabular}{cccccc}
		\begin{tikzpicture}[scale=1*0.8]
		      \node[vertex] (3) at (0.5, -0.866025) {};
		      \node[vertex] (0) at (1., 0.) {};
		      \node[vertex] (4) at (0.5, 0.866025) {};
		      \node[vertex] (5) at (-0.5, 0.866025) {};
		      \node[vertex] (1) at (-1.,  0.) {};
		      \node[vertex] (2) at (-0.5, -0.866025) {};
	   	      \node[vertex] (6) at (0,0) {};
		      \draw[edge] (2)edge(3)  (3)edge(6) (6)edge(2) (1)edge(4) (1)edge(5) (4)edge(0) (0)edge(5);
		      \draw[edge] (2)edge(1) (6)edge(4) (6)edge(5) (0)edge(3);
		\end{tikzpicture}
		&
		\begin{tikzpicture}[scale=1*0.8]
		    \node[vertex] (0) at (-1.00, 0.000) {};
		    \node[vertex] (1) at (1.00, 0.000) {};
		    \node[vertex] (2) at (0.500, 0.866) {};
		    \node[vertex] (3) at (0.500, -0.866) {};
		    \node[vertex] (4) at (-0.500, -0.866) {};
		    \node[vertex] (5) at (-0.500, 0.866) {};
		    \node[vertex] (6) at (0.000, -0.300) {};
		    \node[vertex] (7) at (0.000, 0.300) {};
		    \draw[edge] (1)edge(3) (6)edge(7) (0)edge(5)  ;
		    \draw[edge] (3)edge(4) (3)edge(7) (4)edge(7) (1)edge(6)  ;
		    \draw[edge] (0)edge(6) (5)edge(7) (2)edge(5) (2)edge(7)  ;
		    \draw[edge] (1)edge(2) (0)edge(4)  ;
		\end{tikzpicture}
		&
		\begin{tikzpicture}[scale=1*0.8]
		    \node[vertex] (1) at (0.500, 0.866) {};
		    \node[vertex] (0) at (-0.500, 0.866) {};
		    \node[vertex] (2) at (-1.00, 0.000) {};
		    \node[vertex] (3) at (1.00, 0.000) {};
		    \node[vertex] (4) at (0.000, -0.433) {};
		    \node[vertex] (5) at (0.500, -0.866) {};
		    \node[vertex] (6) at (-0.500, -0.866) {};
		    \node[vertex] (7) at (0.000, 0.000) {};
		    \draw[edge] (2)edge(6) (1)edge(7)  ;
		    \draw[edge] (0)edge(7) (3)edge(5)  ;
		    \draw[edge] (5)edge(7) (5)edge(6) (6)edge(7) (4)edge(6) (4)edge(7) (1)edge(2) (0)edge(3) (4)edge(5)  ;
		    \draw[edge] (1)edge(3) (0)edge(2)  ;
		\end{tikzpicture}		
		&
		\begin{tikzpicture}[scale=1*0.8]
		    \node[vertex] (0) at (0.500, -0.866) {};
		    \node[vertex] (1) at (-0.500, -0.866) {};
		    \node[vertex] (2) at (0.500, 0.866) {};
		    \node[vertex] (3) at (-0.500, 0.866) {};
		    \node[vertex] (4) at (-1.00, 0.000) {};
		    \node[vertex] (5) at (1.00, 0.000) {};
		    \node[vertex] (6) at (0.300, -0.100) {};
		    \node[vertex] (7) at (-0.300, -0.100) {};
		    \draw[edge] (0)edge(7) (1)edge(4)  ;
		    \draw[edge] (0)edge(5) (1)edge(6)  ;
		    \draw[edge] (3)edge(4) (3)edge(7) (4)edge(7) (5)edge(6) (2)edge(6) (2)edge(5) (0)edge(1)  ;
		    \draw[edge] (2)edge(4) (6)edge(7) (3)edge(5)  ;
		\end{tikzpicture}
		&
		\begin{tikzpicture}[scale=1*0.8]
		    \node[vertex] (0) at (-0.588, -0.809) {};
		    \node[vertex] (1) at (0.588, -0.809) {};
		    \node[vertex] (2) at (0.951, 0.309) {};
		    \node[vertex] (3) at (-0.951, 0.309) {};
		    \node[vertex] (4) at (0.235, 0.324) {};
		    \node[vertex] (5) at (-0.235, 0.324) {};
		    \node[vertex] (6) at (0.000, -0.400) {};
		    \node[vertex] (7) at (0.000, 1.00) {};
		    \draw[edge] (1)edge(2) (0)edge(3)  ;
		    \draw[edge] (0)edge(6) (1)edge(6) (0)edge(1)  ;
		    \draw[edge] (2)edge(4) (5)edge(6) (4)edge(7) (2)edge(7)  ;
		    \draw[edge] (5)edge(7) (3)edge(5) (3)edge(7) (4)edge(6)  ;
		\end{tikzpicture}
		&
		\begin{tikzpicture}[scale=1*0.8]
		    \node[vertex] (0) at (0.000, -0.433) {};
		    \node[vertex] (1) at (-0.500, -0.866) {};
		    \node[vertex] (2) at (0.500, -0.866) {};
		    \node[vertex] (3) at (1.00, 0.000) {};
		    \node[vertex] (4) at (-1.00, 0.000) {};
		    \node[vertex] (5) at (0.000, 0.000) {};
		    \node[vertex] (6) at (0.500, 0.866) {};
		    \node[vertex] (7) at (-0.500, 0.866) {};
		    \draw[edge] (3)edge(7) (4)edge(6)  ;
		    \draw[edge] (0)edge(5) (1)edge(7) (4)edge(7) (3)edge(6) (2)edge(3) (2)edge(6) (1)edge(4)  ;
		    \draw[edge] (5)edge(6) (0)edge(2)  ;
		    \draw[edge] (5)edge(7) (0)edge(1)  ;
		\end{tikzpicture}
		\\
			$Q_1$ & $Q_2$ &$Q_3$ &$Q_4$ &$Q_5$ &$Q_6$ 
	\end{tabular}
	\begin{tabular}{ccccc}
			\begin{tikzpicture}[scale=0.761]
			\node[vertex] (5) at (0.5, -0.866025) {};
			\node[vertex] (4) at (1., 0.) {};
			\node[vertex] (2) at (0.5, 0.866025) {};
			\node[vertex] (1) at (-0.5, 0.866025) {};
			\node[vertex] (7) at (-1.,  0.) {};
			\node[vertex] (6) at (-0.5, -0.866025) {};
			\node[vertex] (3) at (0.8,-1.4) {};
			\node[vertex] (0) at (-0.8,-1.4) {};
			\draw[edge] (7)edge(0) (3)edge(4);
			\draw[edge]  (6)edge(5) (5)edge(4) (7)edge(4) (7)edge(6) (0)edge(3);
			\draw[edge] (2)edge(4) (2)edge(6);
			\draw[edge] (1)edge(5) (7)edge(1) ;
			\draw[edge] (2)edge(1);
			\draw[edge] (0)edge(6)  (5)edge(3) ;
		\end{tikzpicture}
		&
		\begin{tikzpicture}[scale=0.761]
			\node[vertex] (1) at (-0.8,-1.4) {};
			\node[vertex] (5) at (1., 0.) {};
			\node[vertex] (3) at (-1.,  0.) {};
			\node[vertex] (6) at (-0.5, 0.866025) {};
			\node[vertex] (2) at (0.5, 0.866025) {};
			\node[vertex] (4) at (-0.5, -0.866025) {};
			\node[vertex] (7) at (0.5, -0.866025) {};
			\node[vertex] (0) at (0.8,-1.4) {};
			\draw[edge] (5)edge(3) (2)edge(6) (4)edge(0);
			\draw[edge]  (2)edge(4) (0)edge(1) (5)edge(7);
			\draw[edge] (2)edge(5) (7)edge(4) (3)edge(6);
			\draw[edge] (1)edge(6) (7)edge(6) (3)edge(4) (1)edge(7) (5)edge(0);
		\end{tikzpicture}
		&
		\begin{tikzpicture}[scale=0.761]
			\node[vertex] (1) at (-0.8,-1.4) {};
			\node[vertex] (7) at (1., 0.) {};
			\node[vertex] (3) at (-1.,  0.) {};
			\node[vertex] (6) at (-0.5, -0.866025) {};
			\node[vertex] (2) at (0.5, -0.866025) {};
			\node[vertex] (5) at (-0.5, 0.866025) {};
			\node[vertex] (4) at (0.5, 0.866025) {};
			\node[vertex] (0) at (0.8,-1.4) {};
			\draw[edge] (7)edge(3) (2)edge(6) (5)edge(1);
			\draw[edge]  (2)edge(5) (7)edge(4) (0)edge(6);
			\draw[edge] (2)edge(7) (4)edge(5) (3)edge(6);
			\draw[edge] (4)edge(6) (3)edge(5) (7)edge(1) (0)edge(7);
			\draw[edge] (0)edge(1);
		\end{tikzpicture}
		&
		\begin{tikzpicture}[scale=1]
			\node[vertex] (a1) at (0.5, -0.866025) {};
			\node[vertex] (a2) at (1., 0.) {};
			\node[vertex] (a3) at (0.5, 0.866025) {};
			\node[vertex] (a4) at (-0.5, 0.866025) {};
			\node[vertex] (a5) at (-1.,  0.) {};
			\node[vertex] (a6) at (-0.5, -0.866025) {};
			\pgfmathsetmacro\hundredFiftycos{cos(150)};
			\pgfmathsetmacro\hundredFiftysin{sin(150)};
			\node[vertex] (a7) at (0.6*\hundredFiftycos, 0.6*\hundredFiftysin) {};
			\node[vertex] (a8) at (1.1*\hundredFiftycos, 1.1*\hundredFiftysin) {};
			\draw[edge] (a1)edge(a4) (a2)edge(a3) (a3)edge(a6) (a6)edge(a5) (a2)edge(a5);
			\draw[edge] (a5)edge(a4) (a6)edge(a1) (a3)edge(a4) (a2)edge(a1);
			\draw[edge] (a5)edge(a7) (a5)edge(a8) (a7)edge(a4) (a7)edge(a8) (a4)edge(a8);
		\end{tikzpicture}

		&
		\begin{tikzpicture}[scale=1.146]
		    \node[vertex] (0) at (0.000, 0.700) {};
		    \node[vertex] (1) at (-0.588, -0.809) {};
		    \node[vertex] (2) at (0.588, -0.809) {};
		    \node[vertex] (3) at (0.588, 0.309) {};
		    \node[vertex] (4) at (-0.588, 0.309) {};
		    \node[vertex] (5) at (-0.235, 0) {};
		    \node[vertex] (6) at (0.000, -0.400) {};
		    \node[vertex] (7) at (0.235, 0) {};
		    \draw[edge] (0)edge(3) (0)edge(4) (0)edge(5) (1)edge(2) (1)edge(4) (1)edge(6) (2)edge(3) (2)edge(6) (3)edge(7) (4)edge(7) (5)edge(6) (5)edge(7) (6)edge(7)  ;
		\end{tikzpicture} \\
		$S_1$ & $S_2$ & $S_3$ & $S_4$ & $S_5$
	\end{tabular}
	\caption{Maximal(w.r.t.\ being a spanning subgraph) non-bipartite movable graphs with a spanning Laman subgraph, at most 8~vertices and no vertex of degree two.}
	\label{fig:constDistClosures}
\end{figure}
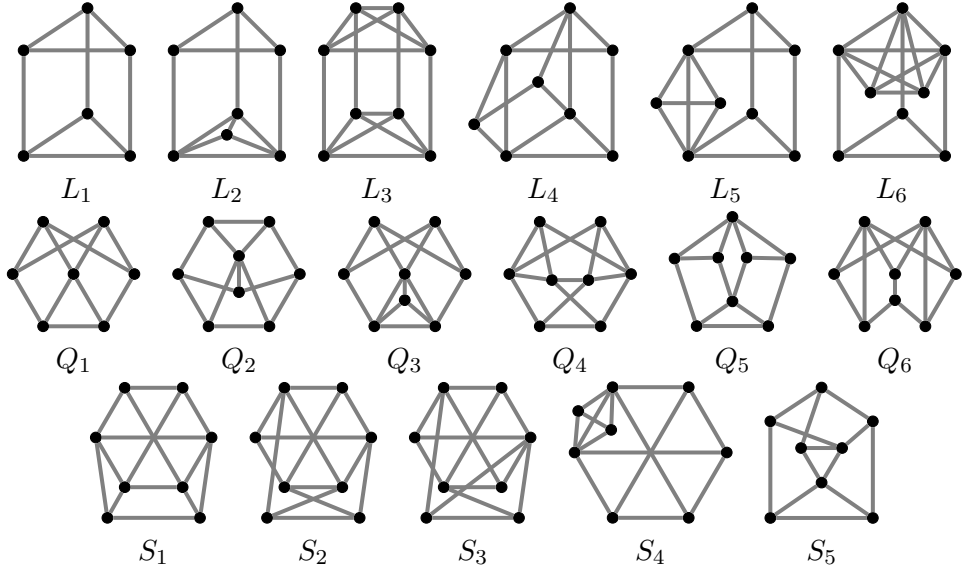
The methods presented in this paper allow to find conditions on the motions of movable graphs.
As examples we present the classifications of $L_i$ for all $i$, $K_{3,3}$ and $Q_1$.

As in \cite{flexibleLabelings,movableGraphs} we want
to transform the equations of \eqref{eq:mainSystemOfEquations} to new ones, where the sum of squares changes to a single product.
We then consider the equations in a complex function field and investigate valuations on the factors.
\begin{definition}
  An irreducible algebraic curve $\C$ in the zero set of \eqref{eq:mainSystemOfEquations}
  is called an \emph{algebraic motion of $(G,\lambda)$ (w.r.t.\ an edge $\bar{u}\bar{v}$)}. 
	For every $u,v\in V_G$ such that $uv\in E_G$, we define $W_{u,v}, Z_{u,v}$ in the complex function field $F(\C)$ by
	\begin{align*}
		W_{u,v}&=(x_v-x_u) + \ci (y_v-y_u)\,, \\
		Z_{u,v}&=(x_v-x_u) - \ci (y_v-y_u).	
	\end{align*}
\end{definition}
Note that $W_{u,v}=-W_{v,u}$ and $Z_{u,v}=-Z_{v,u}$, i.e., they depend on the order of~$u,v$.
Recall that a valuation $\nu:F(\C)\setminus\{0\} \rightarrow \ZZ$ has the properties
\begin{enumerate}
  \item $\nu(W\cdot Z)= \nu(W) + \nu(Z)$ for all $W,Z \in F(\C)\setminus\{0\}$, and
  \item $\nu(W+Z)\geq \min\{\nu(W), \nu(Z)\}$ for all $W,Z \in F(\C)\setminus\{0\}$ such that $W+Z\neq 0$.
\end{enumerate}
We consider only valuations trivial on $\CC$, i.e., $\nu(\CC)=\{0\}$. 
Hence $\nu(W_{u,v})=\nu(W_{v,u})$, which allows to write $\nu(W_e), \nu(Z_e)$ for $e\in E_G$.
Using \eqref{eq:mainSystemOfEquations}, we have
\begin{align}\label{eq:WZequationsLambda} 
	W_{\bar{u},\bar{v}}&=\lambda_{\bar{u}\bar{v}}\,,\quad Z_{\bar{u},\bar{v}}=\lambda_{\bar{u}\bar{v}}\,, \nonumber \\
	W_{u,v}Z_{u,v}&= \lambda_{uv}^2 \quad \text{ for all } uv \in E_G\,.
\end{align}
The following equations hold for every cycle $(u_1, \dots ,u_n, u_{n+1}=u_1)$ in $G$ by the definition of $W_{u,v}$ and $Z_{u,v}$: 
\begin{equation} \label{eq:equationCycles}
	\sum_{i=1}^n W_{u_i, u_{i+1}} =0\,, \qquad	\sum_{i=1}^n Z_{u_i, u_{i+1}} =0\,.
\end{equation}

The valuations of the $W_e$ and $Z_e$ can be used to show the relation between graphs with flexible labelings
and so-called NAC-colorings.
\begin{definition}
	Let~$G$ be a graph. A coloring of edges $\delta\colon  E_G\rightarrow \{\text{\blue{}, \red{}}\}$ 
	is called a \emph{NAC-coloring},
	if it is surjective and for every cycle in $G$,
	either all edges have the same color, or
	there are at least 2 edges in each color.
	The set of all NAC-colorings of~$G$ is denoted by $\nac{G}$.
	NAC-colorings $\delta, \conj{\delta} \in \nac{G}$ are called \emph{conjugate}
	if $\delta(e)\neq\conj{\delta}(e)$ for all $e\in E_G$.
	Figure~\ref{fig:nac} shows examples.
\end{definition}
\begin{figure}[ht]
  \centering
  \begin{tikzpicture}
    \node[vertex] (a) at (0,0) {};
    \node[vertex] (b) at (1,0) {};
    \node[vertex] (c) at (1,1) {};
    \node[vertex] (d) at (0,1) {};
    \draw[bedge] (a)edge(b) (c)edge(d);
    \draw[redge] (b)edge(c) (d)edge(a);
  \end{tikzpicture}
  \begin{tikzpicture}
    \node[vertex] (a) at (0,0) {};
    \node[vertex] (b) at (1,0) {};
    \node[vertex] (c) at (1,1) {};
    \node[vertex] (d) at (0,1) {};
    \draw[redge] (a)edge(b) (c)edge(d);
    \draw[bedge] (b)edge(c) (d)edge(a);
  \end{tikzpicture}
  \qquad\qquad
  \begin{tikzpicture}
    \node[vertex] (a) at (0,0) {};
    \node[vertex] (b) at (1,0) {};
    \node[vertex] (c) at (1,1) {};
    \node[vertex] (d) at (0,1) {};
    \draw[bedge] (a)edge(b);
    \draw[redge] (c)edge(d) (b)edge(c) (d)edge(a);
  \end{tikzpicture}
  \caption{Two conjugate NAC-colorings of the 4-cycle graph on the left and an incorrect coloring on the right.}
  \label{fig:nac}
\end{figure}
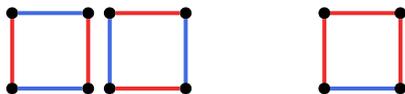

We refer to \cite{flexibleLabelings,movableGraphs} for more details on NAC-colorings.
The following result is proven in~\cite{flexibleLabelings}.
\begin{theorem}\label{thm:nacflexible}
	A connected graph with at least one edge has a flexible labeling if and only if it has a NAC-coloring.
\end{theorem}
Some NAC-colorings can be assigned to an algebraic motion using valuations.
\begin{definition}
  Let $\C$ be an algebraic motion of $(G,\lambda)$.
  A NAC-coloring $\delta$ of $G$ is called \emph{active w.r.t.\ $\C$} if 
  there exists a valuation $\nu$ of $F(\C)$ and $\alpha\in\QQ$ such that for all $uv \in E_G$:
  \begin{equation}
  \label{eq:activeNAC}
  	\delta(uv)= \begin{cases}
  		\red{} &\text{if } \nu(W_{u,v})>\alpha \\
  		\blue{} &\text{otherwise.}
  	\end{cases} 
  \end{equation}
  The set of all active NAC-colorings w.r.t.~$\C$ is denoted by $\nacC{G}{\C}$.
\end{definition}
Actually, any surjective coloring obtained from a valuation using~\eqref{eq:activeNAC} is a NAC-coloring~\cite{flexibleLabelings}.

Let $\Lambda_G\subset \RR^{E_G}$ be the set of 
all proper flexible labelings of $G$.
By \emph{classification of motions}, or \emph{proper flexible labelings} of $G$
we mean the decomposition of the Zariski closure of $\Lambda_G$ into irreducible algebraic sets
$\Lambda_1, \dots\, \Lambda_k$.
Our goal is to provide equations defining the irreducible varieties and an instance for each of them
that is proper flexible.

Clearly, every proper flexible labeling is in some $\Lambda_i$,
but not every $\lambda\in\Lambda_i$ is flexible -- for instance,
it is not guaranteed that it is realizable over $\RR$,
since this would require also inequalities.
Notice also that a labeling in $\Lambda_i$ does not have all edge lengths necessarily positive,
but as long as they are not zero, the system~\eqref{eq:mainSystemOfEquations} does not change
due to taking squares, with the exception of the fixed edge --- switching the sign of
the fixed edge rotates the compatible realizations around the origin by $\pi$.
There also might be a proper subvariety containing flexible labelings that are not proper.

We conclude this section with a remark on genericity,
although we do not need a precise definition as we always speak about flexibility of a fixed labeling.
Algebraic geometry defines a property to be generic on a set $S$
if there is an algebraic subset~$X$ of $S$ of lower dimension such that the property 
holds for all elements of $S\setminus X$.
Our goal to classify all proper flexible labelings of a generically rigid graph $G$
can be viewed as describing the set $X$ of ``non-generic'' realizations.
We know that a realization is ``generic'', hence rigid, if it is injective and
induces a labeling avoiding $\Lambda_G$.
In Section~\ref{sec:consistency} we show that the only kind of flexible labeling 
for the three-prism graph in Figure~\ref{fig:threeprism} is when all 4-cycles are parallelograms.
This explains why the figure on the left shows a generic realization.
Notice that another widely used approach to define a generic realization is to require
that the coordinates are algebraically independent. In this sense,
the left realization in Figure~\ref{fig:threeprism} is not generic though it is rigid.

\section{Leading coefficients system}
\label{sec:comparingLC}
If a graph $G$ is spanned by a Laman graph and there is an algebraic motion of $(G,\lambda)$, 
then the edge lengths $\lambda$ must be non-generic.
In this section, we introduce a method deriving some algebraic equation(s) for $\lambda$.
In general, the method assumes a valuation of the function field of the algebraic motion,
but for certain active NAC-colorings,
all needed information can be recovered from the NAC-coloring itself.

Let $\C$ be an algebraic motion.
Let $\nu$ be a valuation yielding an active NAC-coloring of~$\C$ for some threshold~$\alpha$ using~\eqref{eq:activeNAC},
i.e.\ $|\{\nu(W_e)\colon e \in E_G\}|\geq 2$ as NAC-colorings are required to be surjective.
There is a parametrization of $\C$ such that $W_{u,v}$ 
and~$Z_{u,v}$ can be expressed as Laurent series in the parameter $t$
such that $\ord(W_{u,v})=\nu(W_{u,v})$ (see for instance~\cite{EnglerPrestel2005}).
We denote by $w_{u,v}$, resp.\ $z_{u,v}$, the leading coefficients of 
$W_{u,v}$, resp.\ $Z_{u,v}$ for all $uv\in E_G$.
Clearly, $w_{u,v}=-w_{v,u}$ and $z_{u,v}=-z_{v,u}$.
From the edge equations~\eqref{eq:WZequationsLambda} we have
\begin{align*}
	\lambda_{uv}^2 = W_{u,v} Z_{u,v} 
		= (w_{u,v}t^{\ord W_{u,v}} + \hot )(z_{u,v}t^{-\ord W_{u,v}} + \hot )\,,
\end{align*}
where \hot\ means \emph{higher order terms}. Hence, by expanding and comparing leading coefficients, i.e., setting $t=0$,
we have
\begin{equation}
	\label{eq:lcLambda}
	w_{u,v} z_{u,v} = \lambda_{uv}^2 \quad \text{ for all } uv \in E_G\,.
\end{equation}
The cycle conditions \eqref{eq:equationCycles} yield for every cycle $(u_1, \dots ,u_n, u_{n+1}=u_1)$ in $G$ the equations
\begin{equation*}
	\sum_{i\in\{1,\dots, n\}} (w_{u_i,u_{i+1}}t^{\ord W_{u_i,u_{i+1}}} + \hot ) = 0 =
	\sum_{i\in\{1,\dots, n\}} (z_{u_i,u_{i+1}}t^{-\ord W_{u_i,u_{i+1}}} + \hot )\,.
\end{equation*}
Comparing leading coefficients gives
\begin{equation}
	\label{eq:lcWZcycle}
	\sum_{i= \argmin_{j\in\{1,\dots, n\}} (\ord W_{u_j u_{j+1}})} \hspace{-1.5em} w_{u_i, u_{i+1}} =
	\sum_{i= \argmax_{j\in\{1,\dots, n\}} (\ord W_{u_j u_{j+1}})} \hspace{-1.5em} z_{u_i, u_{i+1}} = 0\,,
\end{equation}
Now, we eliminate $w_{u,v}$ and $z_{u,v}$ for all $uv \in E_G$,
e.\,g.\ by Gr\"obner basis computation taking the $\lambda_{uv}$ to be variables as well.
Assuming that $G$ is spanned by a Laman graph, counting parameters shows that 
we can expect to get at least one algebraic equation in $\lambda_{uv}$ for~$uv\in E_G$.

We illustrate the method on NAC-colorings satisfying specific assumptions.
\begin{proposition}
	\label{prop:orthogonalDiagonals}
	Let $\C$ be an algebraic motion of $(G, \lambda)$.
	Let $H = (v_1, v_2, v_3, v_4)$ be a 4-cycle subgraph of $G$.
	Let $\delta\in\nacC{G}{\C}$ be 
	such that $H$ is \blue{} and there exist \red{}
	paths $P_1$ and $P_2$ from $v_1$ to $v_3$ and from $v_2$ to~$v_4$ (compare Figure~\ref{fig:NACimpliesOrthoDiag}).
	Then $\lambda_{v_1v_2}^2 + \lambda_{v_3v_4}^2 = \lambda_{v_2v_3}^2 + \lambda_{v_1v_4}^2$.
	In~particular, the 4-cycle~$H$ has perpendicular diagonals
	in every injective realization in $\C$.
	
	\begin{figure}[htb]
	  \centering
	    \begin{tikzpicture}
				\node[vertex,label={[labelsty]below:$v_1$},indicatededge=east,rotate=0] (v1) at (0,0) {};
				\node[vertex,label={[labelsty,label distance=2pt]90:$v_2$},indicatededge=north,rotate=90] (v2) at (1.5,-1) {};
				\node[vertex,label={[labelsty]above:$v_3$},indicatededge=west,rotate=180] (v3) at (3,0) {};
				\node[vertex,label={[labelsty]above:$v_4$},indicatededge=south,rotate=-90] (v4) at (1.5,1) {};
				
				\draw[bedge] (v1)edge(v2) (v2)edge(v3) (v3)edge(v4) (v4)edge(v1);
				\draw[arbitrarygraph,colR] ($(v1)!0.5!(v3)$) ellipse (1.8cm and 0.2cm);
				\draw[arbitrarygraph,colR] ($(v2)!0.5!(v4)$) ellipse (0.15cm and 1.5cm);
			\end{tikzpicture}
	  \caption{NAC-coloring assumed in Proposition~\ref{prop:orthogonalDiagonals}.
	  	Ellipses in \red{} indicate red components.}
	  \label{fig:NACimpliesOrthoDiag}
  \end{figure}
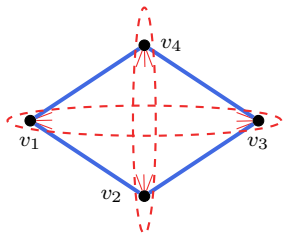
  
\end{proposition}
\begin{proof}
	Let $\nu$ be a valuation making $\delta$ active.
	If a sum in a function field is zero,
	then there are at least two summands with the same valuation,
	namely the minimal one~\cite[Lecture~3]{Deuring}.
	Therefore, we have that $\nu(W_{v_1,v_2})=\nu(W_{v_2,v_3})$ 
	due to the cycle consisting of $P_1$ and the edges $v_1v_2$ and $v_2v_3$
	(blue edges have lower valuation than red ones in active NAC-colorings).
	Similarly we get $\nu(W_{v_1,v_4})=\nu(W_{v_4,v_3})$, $\nu(W_{v_2,v_1})=\nu(W_{v_1,v_4})$
	and $\nu(W_{v_2,v_3})=\nu(W_{v_3,v_4})$.
	Therefore, the equations \eqref{eq:lcWZcycle} give
	\begin{align*}
		w_{v_1,v_2} + w_{v_2,v_3} = w_{v_1,v_4} - w_{v_3,v_4} = w_{v_2,v_3} + w_{v_3,v_4} = w_{v_1,v_2} - w_{v_1,v_4} &= 0\,, \\
		z_{v_1,v_2} + z_{v_2,v_3} + z_{v_3,v_4} - z_{v_1,v_4} &= 0\,.
	\end{align*}
	The first line gives $w_{v_1,v_2}=- w_{v_2,v_3} = w_{v_3,v_4} =  w_{v_1,v_4}$.
	Hence, multiplying the second line by $w_{v_1,v_2}$ and combining it with \eqref{eq:lcLambda},
	yields $\lambda_{v_1v_2}^2 + \lambda_{v_3v_4}^2 = \lambda_{v_2v_3}^2 + \lambda_{v_1v_4}^2$.
	Perpendicularity of diagonals follows from the fact that
	a quadrilateral has orthogonal diagonals if and only if the sums of squares of opposite sides are equal.
\end{proof}

\begin{proposition}
	\label{prop:degenerateTriangle}
  Let $G$ be a graph with a flexible labeling $\lambda$ and an active
  NAC-coloring~$\delta$ for a given algebraic motion $\C$.
  Assume that $G$ has a 3-cycle $(u_1,u_2,u_3)$ and a 4-cycle
  $(v_1,v_2,v_3,v_4)$ both colored in blue, where possibly $v_1$ and $u_1$ might be equal
  and similarly $v_2$ and $u_2$.
  Furthermore $v_1$ and $u_1$, $v_3$ and $u_3$ as well as $u_2$, $v_2$ and $v_4$
  are in the same red components (i.\,e.\ connected monochromatic subgraph in red) respectively, see Figure~\ref{fig:NACimpliesDegTriangle}.
  Then the 3-cycle is a degenerate triangle, namely, the vertices $u_1,u_2,u_3$
  are collinear in all realizations in $\C$, or
  $\lambda_{v_1v_2}= \lambda_{v_1v_4}$ and $\lambda_{v_2v_3}= \lambda_{v_3v_4}$.
  
  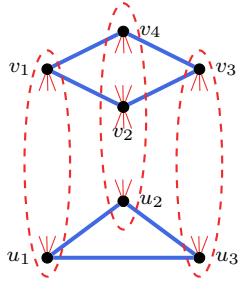
\begin{figure}[htb]
	  \centering
	    \begin{tikzpicture}[xscale=2]
				\node[vertex,label={[labelsty]above:$u_1$},indicatededge=north,rotate=90] (u1) at (0,0) {};
				\node[vertex,label={[labelsty]below:$u_2$},indicatededge=north,rotate=90] (u2) at (0.5,0.75) {};
				\node[vertex,label={[labelsty]below:$u_3$},indicatededge=north,rotate=90] (u3) at (1,0) {};
				\node[vertex,label={[labelsty]below:$v_1$},indicatededge=south,rotate=-90] (v1) at (0,2.5) {};
				\node[vertex,label={[labelsty,label distance=2pt]0:$v_2$},indicatededge=south,rotate=-90] (v2) at (0.5,2) {};
				\node[vertex,indicatededge=north,rotate=90] at (v2) {};
				\node[vertex,label={[labelsty]above:$v_3$},indicatededge=south,rotate=-90] (v3) at (1,2.5) {};
				\node[vertex,label={[labelsty]above:$v_4$},indicatededge=south,rotate=-90] (v4) at (0.5,3) {};
				
				\draw[bedge] (u1)edge(u2) (u1)edge(u3) (u2)edge(u3);
				\draw[bedge] (v1)edge(v2) (v2)edge(v3) (v3)edge(v4) (v4)edge(v1);
				\draw[arbitrarygraph,colR] ($1/2*(v3)+1/2*(u3)$) ellipse (0.15cm and 1.5cm);
				\draw[arbitrarygraph,colR] ($1/2*(v1)+1/2*(u1)$) ellipse (0.15cm and 1.5cm);
				\draw[arbitrarygraph,colR] ($1/2*(v4)+1/2*(u2)$) ellipse (0.15cm and 1.5cm);
			\end{tikzpicture}
	  \caption{NAC-coloring assumed in Proposition~\ref{prop:degenerateTriangle}.
	  	Ellipses in \red{} indicate red components.}
	  \label{fig:NACimpliesDegTriangle}
  \end{figure}
  
\end{proposition}
\begin{proof}
	We proceed similarly as in the proof of Proposition~\ref{prop:orthogonalDiagonals}.
	Let $P_i$ be a \red{} path connecting vertices $u_i$ and $v_i$ for
	$i\in\{1,2,3\}$, $P_1$ or $P_2$ possibly empty,
	and $P_4$ a \red{} path connecting $v_2$ and $v_4$.
	Let $\nu$ be a valuation yielding $\delta$.
	Since the minimum valuation is attained at least twice in a cycle
	and $\nu(W_{u,v})$ of a \red{} edge $uv$ is strictly greater than for blue edges,
	we have that $\nu(W_{u_1,u_2})=\nu(W_{v_1,v_2})$  due to the cycle consisting of
	$P_1, v_1v_2, P_2, u_2u_1$. 
	Next, $\nu(W_{u_1,u_2})=\nu(W_{u_2,u_3})=\nu(W_{u_1,u_3})=\nu(W_{v_2,v_3})=\nu(W_{v_3,v_4})$
	from the triangle $(u_1,u_2,u_3)$ and the cycles formed by $P_2, u_2u_3, P_3, v_3v_2$ and 
	$v_2v_3, v_3v_4,P_4$. Also $\nu(W_{v_1,v_2})=\nu(W_{v_1,v_4})$ due to $v_1v_2,P_4,v_4v_1$.
	The following equations are obtained using \eqref{eq:lcWZcycle} for various cycles:
  \begin{align*}
		z_{u_1,u_2} + z_{u_2,u_3} - z_{u_1,u_3} = w_{u_1,u_2} + w_{u_2,u_3} - w_{u_1,u_3} &= 0  \text{ from } (u_1,u_2,u_3)\,, \\ 
		w_{u_2,u_3} - w_{v_2,v_3} &= 0              \text{ from } P_2, u_2u_3, P_3, v_3v_2\,, \\
		w_{v_2,v_3} + w_{v_3,v_4} &= 0                \text{ from } v_2v_3, v_3v_4,P_4\,, \\
    	w_{u_1,u_2} - w_{v_1,v_2} &= 0              \text{ from } P_1, v_1v_2, P_2, u_2u_1\,, \\
		w_{v_1,v_2} - w_{v_1,v_4} &= 0               \text{ from } v_1v_2,P_4,v_4v_1\,, \\
		z_{v_1,v_2} + z_{v_2,v_3} + z_{v_3,v_4} - z_{v_1,v_4}  &= 0  \text{ from } (v_1,v_2,v_3,v_4)\,. 
	\end{align*}
	We consider also the equations for edge lengths~\eqref{eq:lcLambda}.
  Using Gr\"obner bases we eliminate all $w$ and $z$ variables and get a single equation in the edge lengths.
  With $r=\lambda_{v_3v_4}^2 - \lambda_{v_2v_3}^2$, this equation is quadratic in $r$:
  \begin{align*}
    \lambda_{u_2u_3}^2(\lambda_{v_1v_2}^2 - \lambda_{v_1v_4}^2)^2 -
    (\lambda_{u_1u_2}^2 - \lambda_{u_1u_3}^2 + \lambda_{u_2u_3}^2)(\lambda_{v_1v_2}^2 - \lambda_{v_1v_4}^2)r + 
    \lambda_{u_1u_2}^2 r^2 = 0\,.
  \end{align*}
  The discriminant of this equation is
  \begin{align*}
    &(\lambda_{u_1u_2} - \lambda_{u_1u_3} - \lambda_{u_2u_3})(\lambda_{u_1u_2} + \lambda_{u_1u_3} - \lambda_{u_2u_3})
    (\lambda_{u_1u_2} - \lambda_{u_1u_3} + \lambda_{u_2u_3})\\
    \cdot&(\lambda_{u_1u_2} + \lambda_{u_1u_3} + \lambda_{u_2u_3})
    (\lambda_{v_1v_2} - \lambda_{v_1v_4})^2(\lambda_{v_1v_2} + \lambda_{v_1v_4})^2\,.
  \end{align*}
  In order to get a non-negative discriminant the triangle inequalities tell us that either the triangle is degenerate
  or $\lambda_{v_1v_2} = \lambda_{v_1v_4}$.
  The latter implies that also $\lambda_{v_3v_4} = \lambda_{v_2v_3}$.
\end{proof}

We assume now that the valuation $\nu$ yields only one active NAC-coloring $\delta$,
i.e., for all choices of the threshold $\alpha$, \eqref{eq:activeNAC} gives either $\delta$,
or all edges having the same color.  
This assumption implies that $\{\nu(W_{u,v})\colon uv \in E_G\}=\{0,\alpha\}$.
Then the equations \eqref{eq:lcWZcycle} yield 
\begin{equation}
	\label{eq:lcCycleNAC}
	\sum_{\substack{i\in\{1,\dots, n\}\\ \delta(u_i u_{i+1})=\blue{}}} \hspace{-1.5em} w_{u_i, u_{i+1}}=0
	\qquad	\text{ and } \qquad
	\sum_{\substack{i\in\{1,\dots, n\}\\ \delta(u_i u_{i+1})=\red{}}} \hspace{-1.5em} z_{u_i, u_{i+1}}=0
\end{equation}
for every cycle $C=(u_1, \dots , u_n)$ in $G$ that is not monochromatic.
For monochromatic cycles, the sums are over all edges in the cycle.

Notice that if $\nu$ yields another active NAC-coloring $\delta'$
for another threshold,
then the set $\{(\delta(e),\delta'(e))\colon e\in E_{G}\}$ has 3 elements.
This motivates the following definition.

\begin{definition}
	Let $G$ be a graph and $N\subseteq\nac{G}$.
	A NAC-coloring $\delta$ is called \emph{singleton} w.r.t.\ $N$
	if $|\{(\delta(e),\delta'(e))\colon e\in E_{G}\}|\,\neq 3$ for all $\delta'\in N$.
	We say just \emph{singleton} if $N=\nac{G}$.
\end{definition}

Therefore, if $\C$ is an algebraic motion of $(G, \lambda)$
and an active NAC-coloring $\delta\in \nacC{G}{\C}$ is a singleton w.r.t.\ $\nacC{G}{\C}$,
then we can apply the procedure described in this section using
equation~\eqref{eq:lcCycleNAC} instead of \eqref{eq:lcWZcycle}.
It does not matter whether we apply the procedure 
with a NAC-coloring or its conjugate, since it just
corresponds to swapping $z_{u,v}$ and $w_{u,v}$.

\section{Active NAC-colorings of motions restricted to subgraphs}
\label{sec:ramification}
In this section we exploit relations between active NAC-colorings of an algebraic motion of a graph
and active NAC-colorings of motions obtained by restrictions to (4-cycle) subgraphs.

A \emph{$\mu$-number} is a more precise characteristic of a NAC-coloring than being active.
Let $\Val{\C}$ denote all valuations of the function field 
$F(\C)$ surjective on $\ZZ$ and trivial on $\CC$.
\begin{definition}
	Let $\C$ be an algebraic motion of $(G,\lambda)$.
	For $\delta\in\nac{G}$ and a valuation~$\nu$, let
	\begin{equation*}
		\gap(\delta,\nu):=\max \left\{0,\, \min_{e\in E_G}\{\nu(W_e):\delta(e)=\red{}\}-\max_{e\in E_G}\{\nu(W_e):\delta(e)=\blue{}\} \right\}.
	\end{equation*}
	We define a $\mu$-number 
	\begin{equation*}
		\mu(\delta,\C)=\sum_{\nu \in \Val{\C}} \gap(\delta,\nu)\,.
	\end{equation*}
\end{definition}
Since there are only finitely many valuations 
$\nu\in \val(\C)$ such that $\nu(X)\neq 0$ for a given $X\in F(\C)\setminus\{0\}$ (see for instance \cite[Lecture 6]{Deuring}),
there are only finitely many valuations $\nu\in \Val{\C}$ 
such that $\gap(\delta,\nu)\neq 0$. Hence,
we have $\mu(\delta, \C)\in \NN_0$.

If $\delta$ is an active NAC-coloring due to a valuation $\nu$ and threshold $\alpha$,
then we have $\nu(W_e)> \alpha \geq \nu(W_e')$ for all $e,e'$ such that $\delta(e)=\red{}$ and $\delta(e')=\blue{}$,
hence ${\gap(\delta,\nu)>0}$.
Otherwise,
there must be edges~$e, e'$ such that $\delta(e)=\red{}$, $\delta(e')=\blue{}$ and $\nu(W_e)<\nu(W_{e'})$.
Then $\gap(\delta,\nu)=0$. These observations give the following remark.
\begin{remark}
	\label{rem:mu_active}
	The set of active NAC-colorings of an algebraic motion $\C$ satisfies
	\begin{equation*}
		\nacC{G}{\C}=\{\delta\in\nac{G}\,\colon \mu(\delta,\C)\neq 0\}\,.
	\end{equation*}
\end{remark}

Let $\C$ be an algebraic motion of a graph $G$
and $G'$ be a subgraph of $G$.
We recall a few definitions related to maps between algebraic curves and their function fields.
We consider the projection ${f\colon \C \rightarrow \C'}$ to the vertices of $G'$
and assume that $\C'$ is an algebraic motion of~$G'$.
The function field $F(\C)$ is an algebraic extension of the function field~$F(\C')$.
The degree of the map $f$ is defined as $\deg(f) := [F(\C):F(\C')]$,
which is the cardinality of the fiber $f^{-1}(y)$ for a generic point $y\in \C'$.
More about these notions can be found in~\cite{Shafarevich}. 
If $\nu\in\val(\C)$, then $\nu(F(\C'))=r \ZZ$ for some positive integer $r$
and there is a valuation~$\nu'\in\val(\C')$
such that $\nu(x) = r\nu'(x)$ for all $x\in F(\C')$.
We say that $\nu$ \emph{extends} $\nu'$.
The integer $r$ is called \emph{ramification index} of $\nu$ over $F(\C')$,
denoted by $\ram(\nu)$.
Since the residue field of a function field of a curve over complex numbers is always $\CC$,
we have
\begin{equation}\label{eq:ramForm}
	\sum_{\mathclap{\substack{\nu \in \Val{\C} \\ \nu \text{ ext. } \nu'}}}\ram(\nu) = \deg(f)\,,
\end{equation}
where $\nu'\in\Val{\C'}$;
see~\cite{Deuring,EnglerPrestel2005,Goldschmidt2003}.
Now we are ready to prove the main theorem of this section.
\begin{theorem}\label{thm:ramificationFormula}
	Let $\C$ be an algebraic motion of $(G,\lambda)$.
	Let $G'$ be a subgraph of $G$ and 
	$f:\C\rightarrow\C'$ be the projection of $\C$ into an algebraic motion $\C'$ of~$G'$.
	If $\delta'\in\nac{G'}$, then
	\begin{equation}\label{eq:ramificationFormula}
		\sum_{i=1}^k \mu(\delta_i,\C)=\mu(\delta',\C')\cdot \deg(f)\,,
	\end{equation}
	where $\{\delta_1,\dots,\delta_k\} = \{\delta\in\nac{G}\colon \delta|_{E_{G'}}=\delta'\}$. 
\end{theorem}
\begin{proof}
	First, we prove that for every $\nu\in \Val{\C}$ we have
	\begin{equation}\label{eq:sumOfGaps}
		\gap(\delta',\nu|_{F(\C')})=\sum_{i=1}^k \gap(\delta_i,\nu)\,.
	\end{equation} 
	If $\gap(\delta',\nu|_{F(\C')})=0$, 
	then there exist edges $e_1,e_2\in E_{G'}$ such that $\delta'(e_1)=\red{}$,
	$\delta'(e_2)=\blue{}$ and $\nu|_{F(\C')}(W_{e_1})\leq\nu|_{F(\C')}(W_{e_2})$.
	Hence, $\gap(\delta_i,\nu)=0$ for all~$i$ since
	\begin{equation*}
		\min_{e\in E_G}\{\nu(W_e):\delta_i(e)=\red{}\}\leq \nu(W_{e_1})\leq\nu(W_{e_2}) \leq\max_{e\in E_G}\{\nu(W_e):\delta_i(e)=\blue{}\}\,,
	\end{equation*}	
	 and the claim holds.
	Otherwise, let
	\begin{align*}
		\alpha&:= \max_{e\in E_{G'}}\{\nu(W_e)\colon \delta'(e)=\blue{}\}\,, \text{ and } \\
		\beta&:= \min_{e\in E_{G'}}\{\nu(W_e)\colon \delta'(e)=\red{}\}\,.
	\end{align*}
	Let $\alpha=\alpha_0< \alpha_1< \dots < \alpha_n=\beta$ 
	be such that $\{\alpha_0,\dots,\alpha_n\}=[\alpha,\beta]\cap\{\nu(W_e) \colon e \in E_G\}$.
	For $i\in\{0,\dots,n-1\}$, let $\gamma_i$ be the active NAC-coloring obtained from $\nu$ with threshold~$\alpha_i$.
	All NAC-colorings $\gamma_i$ are extensions of $\delta'$, thus $\{\gamma_0, \dots,\gamma_{n-1}\}\subset\{\delta_1,\dots,\delta_k\}$.
	On the other hand, we have for all $j\in\{1,\dots,k\}$ that either $\gap(\delta_j, \nu)=0$, or 
	$\gap(\delta_j, \nu)=\alpha_{i_j+1}-\alpha_{i_j}=\gap(\gamma_{i_j},\nu)$ for some $i_j\in\{0,\dots,n-1\}$.
	In the latter case,~$\delta_j=\gamma_{i_j}$.
	Figure~\ref{fig:gaps} illustrates the two cases.
	By this we get
	\begin{equation*}
		\sum_{j=1}^k \gap(\delta_j,\nu)=\sum_{i=0}^{n-1} \gap(\gamma_{i},\nu)=\sum_{i=0}^{n-1} (\alpha_{i+1}-\alpha_{i})=\beta-\alpha=\gap(\delta',\nu|_{F(\C')})\,.
	\end{equation*}
	We can conclude the proof now.
	\begin{align*}
		\sum_{j=1}^k \mu(\delta_j,\C) &=\sum_{j=1}^k \sum_{\nu \in \Val{\C}} \hspace{-0.8em} \gap(\delta_j,\nu) = \hspace{-0.25cm}\sum_{\nu \in \Val{\C}} \sum_{j=1}^k \gap(\delta_j,\nu) \\ 
		&\myeq{\eqref{eq:sumOfGaps}} \sum_{\nu \in \Val{\C}} \hspace{-0.8em}\gap(\delta',\nu|_{F(\C')})
			= \hspace{-0.25cm}\sum_{\nu' \in \Val{\C'}} \sum_{\substack{\nu \in \Val{\C} \\ \nu \text{ ext. } \nu'}}\hspace{-0.8em}\gap(\delta',\nu|_{F(\C')}) \\ 
		&= \hspace{-0.25cm}\sum_{\nu' \in \Val{\C'}} \sum_{\substack{\nu \in \Val{\C} \\ \nu \text{ ext. } \nu'}}\hspace{-0.8em} \gap(\delta', \ram(\nu)\cdot\nu')\\
		&= \hspace{-0.25cm}\sum_{\nu' \in \Val{\C'}} \hspace{-1em} \gap(\delta', \nu')  \sum_{\mathclap{\substack{\nu \in \Val{\C} \\ \nu \text{ ext. } \nu'}}}\ram(\nu) \\
		&\myeq{\eqref{eq:ramForm}}\hspace{-0.25cm}\sum_{\nu' \in \Val{\C'}} \hspace{-1em} \gap(\delta', \nu') \cdot \deg(f) = \deg(f)\cdot \mu(\delta', \C')   \,.\qedhere
	\end{align*}
	
	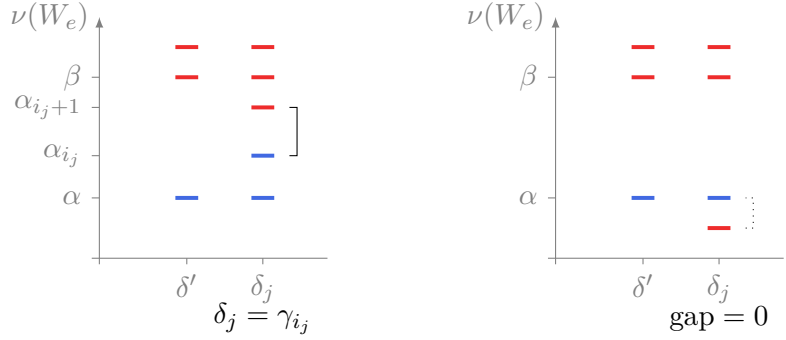
\begin{figure}[ht]
		\centering
			\begin{tikzpicture}[yscale=0.8]
				\begin{scope}
					\draw[axes] (0,-0.1) -- (0,4) node[anchor=east] {$\nu(W_e)$};
					\draw[gridl] (-0.1,0) -- (3,0);
					\draw[gridl] (-0.1,1) node[anchor=east] {$\alpha$}-- (0,1);
					\draw[gridl] (-0.1,3) node[anchor=east] {$\beta$}-- (0,3);
					\draw[gridl] (-0.1,2.5) node[anchor=east] {$\alpha_{i_j+1}$}-- (0,2.5);
					\draw[gridl] (-0.1,1.7) node[anchor=east] {$\alpha_{i_j}$}-- (0,1.7);
					\draw[gridl] (1.15,-0.1) node[anchor=north] {$\delta'$}-- (1.15,0);
					\draw[gridl] (2.15,-0.1) node[anchor=north] {$\delta_j$}-- (2.15,0);
					\begin{scope}
						\draw[bedge] (1,1) -- (1.3,1);
						\draw[redge] (1,3) -- (1.3,3);
						\draw[redge] (1,3.5) -- (1.3,3.5);
						\draw[bedge] (2,1) -- (2.3,1);
						\draw[bedge] (2,1.7) -- (2.3,1.7);
						\draw[redge] (2,2.5) -- (2.3,2.5);
						\draw[redge] (2,3) -- (2.3,3);
						\draw[redge] (2,3.5) -- (2.3,3.5);
					\end{scope}
					\draw[black,thin] (2.5,2.5) -- (2.6,2.5) -- (2.6,1.7) -- (2.5,1.7);
					\node at (2.15,-1) {$\delta_j=\gamma_{i_j}$};
				\end{scope}
				\begin{scope}[xshift=6cm]
					\draw[axes] (0,-0.1) -- (0,4) node[anchor=east] {$\nu(W_e)$};
					\draw[gridl] (-0.1,0) -- (3,0);
					\draw[gridl] (-0.1,1) node[anchor=east] {$\alpha$}-- (0,1);
					\draw[gridl] (-0.1,3) node[anchor=east] {$\beta$}-- (0,3);
					\draw[gridl] (1.15,-0.1) node[anchor=north] {$\delta'$}-- (1.15,0);
					\draw[gridl] (2.15,-0.1) node[anchor=north] {$\delta_j$}-- (2.15,0);
					\begin{scope}
						\draw[bedge] (1,1) -- (1.3,1);
						\draw[redge] (1,3) -- (1.3,3);
						\draw[redge] (1,3.5) -- (1.3,3.5);
						\draw[bedge] (2,1) -- (2.3,1);
						\draw[redge] (2,0.5) -- (2.3,0.5);
						\draw[redge] (2,3) -- (2.3,3);
						\draw[redge] (2,3.5) -- (2.3,3.5);
					\end{scope}
					\draw[black,thin,dotted] (2.5,0.5) -- (2.6,0.5) -- (2.6,1) -- (2.5,1);
					\node at (2.15,-1) {$\gap=0$};
				\end{scope}
			\end{tikzpicture}
		\caption{Gaps for different settings of the edge colorings $\delta_j$ for edges in $G$ (right column)
		 that are not in $G'$ (left column). Horizontal bars represent edges in their color at the level of their valuation.}
		\label{fig:gaps}
	\end{figure}
\end{proof}

As an immediate consequence of Theorem~\ref{thm:ramificationFormula} and Remark~\ref{rem:mu_active} we get that active NAC-colorings of a subgraph 
are the restrictions of active NAC-colorings of the whole graph.
\begin{corollary} \label{cor:muOfExtIsAlsoZero}
	$\nacC{G'}{\C'} = \{\delta|_{E_{G'}}\in \nac{G'} \colon \delta \in \nacC{G}{\C}\}$, with notation as in Theorem~\ref{thm:ramificationFormula}.  
\end{corollary}

In order to classify motions of a graph,
we use restrictions to 4-cycles.
Now we focus on types of NAC-colorings and motions of a 4-cycle.
 \begin{definition}
	Let $C_4=(v_1,v_2,v_3,v_4)$ be a 4-cycle.
	The \emph{type} of $\delta\in\nac{C_4}$ is
	\begin{itemize}
		 \item $O$ (\,$\colO{}$) if $\delta(v_1v_2)=\delta(v_3v_4)$, or
		 \item $L$ (\,$\colL{}$) if $\delta(v_1v_2)=\delta(v_1v_4)$, or
		 \item $R$ (\,$\colR{}$) otherwise, see also Figure~\ref{fig:quadrilateral}.
	\end{itemize}
\end{definition}
The terminology comes from the edge with the same color as the bottom one
if the vertices of the 4-cycle are numbered counterclockwise starting with the bottom left,
see Figure~\ref{fig:quadrilateral}.
 
 \begin{figure}[ht]
  \centering
  	\begin{tabular}{cc}
  	\begin{tikzpicture}[scale=0.9, baseline=0.5cm]
		    \node[vertex] (a) at (0,0) {};
			\node[vertex] (b) at (2,0) {};
			\node[vertex] (c) at (1.8,0.8) {};
			\node[vertex] (d) at (0.2,1) {};
			\draw[edge] (a) to node[below] {$\lambda_{12}$} (b);
			\draw[edge] (b) to node[right] {$\lambda_{23}$} (c);
			\draw[edge] (c) to node[above] {$\lambda_{34}$} (d);
			\draw[edge] (d) to node[left] {$\lambda_{14}$} (a);
			\node[below left=0.1cm] at (a) {$v_1$};
			\node[below right=0.1cm] at (b) {$v_2$};
			\node[above right=0.1cm] at (c) {$v_3$};
			\node[above left=0.1cm] at (d) {$v_4$};
	\end{tikzpicture}
	&
    \begin{tabular}{c|c|c}
		\begin{tikzpicture}[scale=1.]
			\node[vertex] (a) at (0,0) {};
			\node[vertex] (b) at (1,0) {};
			\node[vertex] (c) at (.9,0.4) {};
			\node[vertex] (d) at (0.1,0.5) {};		
			\draw[bedge] (a) to (b);
			\draw[redge] (b) to (c);
			\draw[redge] (c) to (d);
			\draw[bedge] (d) to (a);
		\end{tikzpicture}
		&
		\begin{tikzpicture}[scale=1.]
			\node[vertex] (a) at (0,0) {};
			\node[vertex] (b) at (1,0) {};
			\node[vertex] (c) at (.9,0.4) {};
			\node[vertex] (d) at (0.1,0.5) {};
			\draw[bedge] (a) to (b);
			\draw[redge] (b) to (c);
			\draw[bedge] (c) to (d);
			\draw[redge] (d) to (a);
		\end{tikzpicture}
		&
	    \begin{tikzpicture}[scale=1.]
	     	\node[vertex] (a) at (0,0) {};
			\node[vertex] (b) at (1,0) {};
			\node[vertex] (c) at (.9,0.4) {};
			\node[vertex] (d) at (0.1,0.5) {};
			\draw[bedge] (a) to (b);
			\draw[bedge] (b) to (c);
			\draw[redge] (c) to (d);
			\draw[redge] (d) to (a);
		\end{tikzpicture}
		\\
		\begin{tikzpicture}[scale=1.]
			\node[vertex] (a) at (0,0) {};
			\node[vertex] (b) at (1,0) {};
			\node[vertex] (c) at (.9,0.4) {};
			\node[vertex] (d) at (0.1,0.5) {};		
			\draw[redge] (a) to (b);
			\draw[bedge] (b) to (c);
			\draw[bedge] (c) to (d);
			\draw[redge] (d) to (a);
		\end{tikzpicture}
		&
		\begin{tikzpicture}[scale=1.]
			\node[vertex] (a) at (0,0) {};
			\node[vertex] (b) at (1,0) {};
			\node[vertex] (c) at (.9,0.4) {};
			\node[vertex] (d) at (0.1,0.5) {};
			\draw[redge] (a) to (b);
			\draw[bedge] (b) to (c);
			\draw[redge] (c) to (d);
			\draw[bedge] (d) to (a);
		\end{tikzpicture}
		&
		\begin{tikzpicture}[scale=1.]
	     	\node[vertex] (a) at (0,0) {};
			\node[vertex] (b) at (1,0) {};
			\node[vertex] (c) at (.9,0.4) {};
			\node[vertex] (d) at (0.1,0.5) {};
			\draw[redge] (a) to (b);
			\draw[redge] (b) to (c);
			\draw[bedge] (c) to (d);
			\draw[bedge] (d) to (a);
		\end{tikzpicture}
		\\
		$L = \colL$ & $O = \colO$ & $R = \colR$
	\end{tabular}
	\end{tabular}
  \caption{Labeling of the 4-cycle $C_4$ and notation for conjugated NAC-colorings.}
  \label{fig:quadrilateral}
\end{figure}
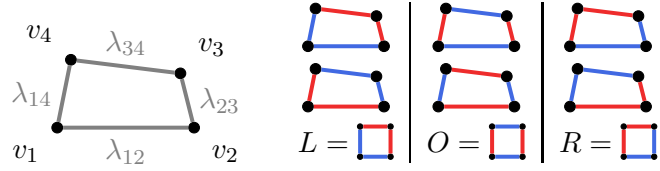

Table~\ref{tab:colQuadrilateral} classifies the motions of $C_4$ with various labelings $\lambda$ 
according to the types of active NAC-colorings.
It is important to stick to the counterclockwise numbering starting at bottom left in order to have the types well defined.
The computations in Appendix~\ref{sec:NACsQuadrilateral} show how motions determine NAC-colorings.
Since the table is complete we also know that active NAC-colorings do imply a motion.
There might be different algebraic motions compatible 
with the same $\lambda$, since
the zero set of~\eqref{eq:mainSystemOfEquations} is not necessarily irreducible.
An \emph{odd} deltoid has a degenerate component,
where two vertices $v_1$ and $v_3$ coincide, similarly \emph{even} deltoid and rhombus.
If the opposite edge lengths are equal,
then there are two non-degenerate motions: parallel and antiparallel.

\begin{table}[htb]
	\centering
	\begin{tabular}{lcCccc}
	\textbf{Quadrilateral} & \textbf{Motion}  &   \multicolumn{3}{c}{\textbf{Types of $\nacC{C_4}{\C}$}} & \textbf{Equations}	\\[8pt]
  \multirow{2}{*}{Rhombus}&parallel &\caseP	&	\colO & $O$ &  
	 		\multirow{2}{*}[0pt]{$\left\{\begin{aligned}\lambda_{12}&=\lambda_{23}=\\ \lambda_{34}&=\lambda_{14}\end{aligned}\right\}$} \\[5pt]
	 & degenerate	& &	\colL{} resp. \colR  & $L$ resp.\ $R$ & \\
	        \midrule
	Parallelogram & parallel &\caseP	&		\colO & $O$ & 
			\multirow{2}{*}{$\left\{\begin{aligned}\lambda_{12}&=\lambda_{34},\\ \lambda_{23}&=\lambda_{14}\end{aligned}\right\}$} \\[5pt]
	Antiparallelogram      &  antiparallel  &\caseA	&	\colL \quad \colR& $L, R$ & \\ \midrule
	\multirow{2}{*}{Deltoid (even)} & non-degenerate	& \caseE &		\colO  \quad \colR	& $O, R$ & 
			\multirow{2}{*}{$\left\{\begin{aligned}\lambda_{12}&=\lambda_{14}, \\ \lambda_{23}&=\lambda_{34}\end{aligned}\right\}$} \\[5pt]
	               & degenerate	& &	\colL  & $L$ &  \\ \midrule
	\multirow{2}{*}{Deltoid (odd)} & non-degenerate	&\caseO	&	\colO \quad \colL & $O, L$ & 
			\multirow{2}{*}{$\left\{\begin{aligned}\lambda_{12}&=\lambda_{23},\\ \lambda_{34}&=\lambda_{14}\end{aligned}\right\}$} \\[5pt]
	          &degenerate	& &	\colR& $R$ & \\ \midrule
	General	&	&  \caseG &  \colO \quad \colL \quad \colR  		& $O, L, R$ & otherwise
	\end{tabular}
	\vspace{0.25cm}
	\caption{Active NAC-colorings of possible motions $\C$ of a $(C_4,\lambda)$.}
	\label{tab:colQuadrilateral}
\end{table}

It is well known that the graphs $L_i$ (Figure~\ref{fig:constDistClosures}) are movable by making the vertical edges 
in the figure parallel and same lengths. 
Looking at 4-cycles shows that this is the only option.

\begin{corollary}
	Let $i\in\{1,\dots, 6\}$.
	If $\lambda$ is a proper flexible labeling of $L_i$,
	then every 4-cycle that is colored nontrivially by $\delta_{i}$
	(see Figure~\ref{fig:graphsOneNAC}), is a parallelogram.
\end{corollary}
\begin{proof}
	Since $\delta_{i}$ is the only NAC-coloring of $L_i$ modulo conjugation,
	it is the only active one in every algebraic motion. 
	As we see in Figure~\ref{fig:graphsOneNAC} the restriction of $\delta_{i}$ to each nontrivially colored 4-cycle is of type O (\colO).
	According to Table~\ref{tab:colQuadrilateral} it must be in a parallel motion.
\end{proof}
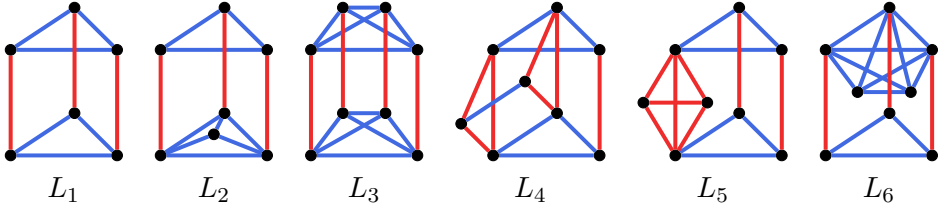
\begin{figure}[htb]
	\centering
	\begin{tabular}{cccccc}
	\begin{tikzpicture}[scale=1.4]
	    \node[vertex] (0) at (0.60, 0.40) {};
	    \node[vertex] (1) at (1, 0) {};
	    \node[vertex] (2) at (1, 1) {};
	    \node[vertex] (3) at (0.60, 1.4) {};
	    \node[vertex] (6) at (0, 1) {};
	    \node[vertex] (7) at (0, 0) {};
	    \draw[bedge] (0)edge(1) (0)edge(7) (1)edge(7) (2)edge(3) (2)edge(6) (3)edge(6)  ;
	    \draw[redge] (0)edge(3) (1)edge(2)  (6)edge(7)  ;
	\end{tikzpicture}
	&
	\begin{tikzpicture}[scale=1.4]
	    \node[vertex] (0) at (0.60, 0.40) {};
	    \node[vertex] (1) at (1, 0) {};
	    \node[vertex] (2) at (1, 1) {};
	    \node[vertex] (3) at (0.60, 1.4) {};
	    \node[vertex] (4) at (0.50, 0.2) {};
	    \node[vertex] (6) at (0, 1) {};
	    \node[vertex] (7) at (0, 0) {};
	    \draw[bedge] (0)edge(1) (0)edge(7) (1)edge(7) (2)edge(3) (2)edge(6) (3)edge(6) (4)edge(0) (4)edge(1) (4)edge(7)  ;
	    \draw[redge] (0)edge(3) (1)edge(2)  (6)edge(7)  ;
	\end{tikzpicture}
	&
	\begin{tikzpicture}[scale=1.4]
	    \node[vertex] (0) at (0.70, 0.40) {};
	    \node[vertex] (1) at (1, 1) {};
	    \node[vertex] (2) at (0.30, 1.4) {};
	    \node[vertex] (3) at (0, 1) {};
	    \node[vertex] (4) at (0, 0) {};
	    \node[vertex] (5) at (0.30, 0.40) {};
	    \node[vertex] (6) at (1, 0) {};
	    \node[vertex] (7) at (0.70, 1.4) {};
	    \draw[bedge] (0)edge(4) (0)edge(5) (0)edge(6) (1)edge(2) (1)edge(3) (1)edge(7) (2)edge(3) (2)edge(7) (3)edge(7) (4)edge(5) (4)edge(6) (5)edge(6)  ;
	    \draw[redge] (0)edge(7) (1)edge(6) (2)edge(5) (3)edge(4)  ;
	\end{tikzpicture}
	&
	\begin{tikzpicture}[scale=1.4]
	    \node[vertex] (0) at (1, 0) {};
	    \node[vertex] (1) at (-0.30, 0.30) {};
	    \node[vertex] (2) at (0.30, 0.70) {};
	    \node[vertex] (3) at (1, 1) {};
	    \node[vertex] (4) at (0.60, 1.4) {};
	    \node[vertex] (5) at (0, 1) {};
	    \node[vertex] (6) at (0.60, 0.40) {};
	    \node[vertex] (7) at (0, 0) {};
	    \draw[bedge] (0)edge(6) (0)edge(7) (3)edge(4) (3)edge(5) (4)edge(5) (6)edge(7)  ;
	    \draw[redge] (0)edge(3) (1)edge(5) (1)edge(7) (2)edge(4) (2)edge(6) (4)edge(6) (5)edge(7)  ;
	    \draw[bedge] (1)edge(2);
	\end{tikzpicture}
	&
	\begin{tikzpicture}[scale=1.4]
	    \node[vertex] (0) at (0.60, 0.40) {};
	    \node[vertex] (1) at (1, 0) {};
	    \node[vertex] (2) at (1, 1) {};
	    \node[vertex] (3) at (0.60, 1.4) {};
	    \node[vertex] (4) at (0.30, 0.50) {};
	    \node[vertex] (5) at (-0.30, 0.50) {};
	    \node[vertex] (6) at (0, 1) {};
	    \node[vertex] (7) at (0, 0) {};
	    \draw[bedge] (0)edge(1) (0)edge(7) (1)edge(7) (2)edge(3) (2)edge(6) (3)edge(6)  ;
	    \draw[redge] (0)edge(3) (1)edge(2) (4)edge(5) (4)edge(6) (4)edge(7) (5)edge(6) (5)edge(7) (6)edge(7)  ;
	\end{tikzpicture}
	&
	\begin{tikzpicture}[scale=1.4]
	    \node[vertex] (0) at (0.60, 0.40) {};
	    \node[vertex] (1) at (0, 0) {};
	    \node[vertex] (2) at (1, 0) {};
	    \node[vertex] (3) at (0.80, 0.60) {};
	    \node[vertex] (4) at (0.30, 0.60) {};
	    \node[vertex] (5) at (1, 1) {};
	    \node[vertex] (6) at (0, 1) {};
	    \node[vertex] (7) at (0.60, 1.4) {};
	    \draw[bedge] (0)edge(1) (0)edge(2) (1)edge(2) (3)edge(4) (3)edge(5) (3)edge(6) (3)edge(7) (4)edge(5) (4)edge(6) (4)edge(7) (5)edge(6) (5)edge(7) (6)edge(7)  ;
	    \draw[redge] (0)edge(7) (1)edge(6) (2)edge(5)  ;
	\end{tikzpicture}
	\\
	$L_1$ & $L_2$ &$L_3$ &$L_4$ &$L_5$ &$L_6$ 	
	\end{tabular}
	\caption{The only NAC-colorings $\delta_1, \dots, \delta_6$ of $L_1, \dots, L_6$ modulo conjugation.}
	\label{fig:graphsOneNAC}
\end{figure}

We compute the $\mu$-numbers for active NAC-colorings of a 4-cycle
in order to be able to use formula \eqref{eq:ramificationFormula} to construct a system of equations
based on restrictions to 4-cycles.  
\begin{theorem}	\label{thm:muQlatsIsOne}
	Let $\C$ be an algebraic motion of $C_4$.
	If $\delta\in \nacC{C_4}{\C}$, then $\mu(\delta,\C)=1$.
\end{theorem}
\begin{proof}
	There are three possible NAC-colorings of $C_4 = (v_1,v_2,v_3,v_4)$ modulo conjugation.
	W.l.o.g., we assume that the edge $v_1v_2$ is fixed and $\delta(v_1v_2)=\blue{}$.	
	Let $\delta_O, \delta_L$ and~$\delta_R$ be the three possibilities for $\delta$ 
	with types $\colO, \colL$ and $\colR$ respectively (see Figure~\ref{fig:quadrilateral}).
	We set $\mu_O:=\mu(\delta_O,\C), \mu_L:=\mu(\delta_L,\C)$ and $\mu_R:=\mu(\delta_R,\C)$.
	We want to show that if $\mu_O, \mu_L$ or $\mu_R$ is nonzero,
	i.e., the corresponding NAC-coloring is active for $\C$, then it equals 1.

	Let $f\colon \C \rightarrow \C_{f}$, resp.\ $g\colon \C \rightarrow \C_{g}$,
	be the projection of $\C$ into realizations of the subgraph of $C_4$
	induced by removing vertex $v_3$, resp.~$v_4$.
	For these two subgraphs,
	we set $\mu_{f}:=\mu(\colOR,\C_{f})$ and $\mu_{g}:=\mu(\colOL,\C_{g})$.
	Theorem~\ref{thm:ramificationFormula} gives
	\begin{equation} \label{eq:ramFromForQlats}
		\mu_O+\mu_R=\mu_{f}\cdot \deg(f), \qquad \mu_O+\mu_L=\mu_{g}\cdot \deg(g)
	\end{equation}	

	We assume that $v_1$ and $v_3$ do not coincide in $\C_g$, i.e., $\C_g$ is a curve.
	Since $v_1v_2$ is fixed, $\rho(v_1)=(0, 0)$ and~$\rho(v_2)=(\lambda(v_1v_2), 0)$.
	By the definition of $W_{2,3}$ and $Z_{2,3}$, we have $2(x_3-\lambda(v_1v_2))=W_{2,3}+Z_{2,3}=W_{2,3}+\lambda(v_2v_3)^2/W_{2,3}$.
	Hence, $x_3 \in \CC(W_{2,3})$. Similarly, $y_3 \in \CC(W_{2,3})$.
	Thus, the rational function field $\CC(W_{2,3})$ is the function field of~$\C_{g}$.
	But there are only two valuations $\nu_1,\nu_2$ of $\CC(W_{2,3})$ trivial on $\CC$ such that $\nu_i(W_{2,3})\neq 0$
	(compare~\cite[Lecture~3]{Deuring}).
	Those valuations yield $\nu_1(W_{2,3})=1$ and $\nu_2(W_{2,3})=-1$.
	Since $\nu_i(W_{1,2})=\nu_i(\lambda(v_1v_2))=0$,
	we have $\mu_{g}=\gap(\colOL,\nu_1)+\gap(\colOL,\nu_2)= 1+0=1$.
	Analogously, if $v_2$ and $v_4$ do not coincide in $\C_f$,  
	$\mu_{f}=1$.  
	Hence, \eqref{eq:ramFromForQlats} simplifies to 
	\begin{align*}
		\mu_O+\mu_R= \deg(f)\,, \qquad
		\mu_O+\mu_L= \deg(g)\,.
	\end{align*}

	The degree of $g$
	is the number of possibilities
	we have for extending a realization of the subgraph to a realization of $C_4$.
	The vertex $v_4$ in $\C$ lies in the intersection of two circles
	centered at vertices~$v_1$ and $v_3$,
	but one point of the intersection might not be in~$\C$.
	Hence, $\deg(g), \deg(f)\in\{1,2\}$. 
	We treat all non-degenerate motions separately. 
	\begin{itemize}
	  \item If the motion $\C$ is antiparallel, then $\mu_L>0, \mu_R>0$ and $\mu_O=0$.
	  				Only one of the two possible positions of $v_4$, resp.\ $v_3$, yields
	  				an antiparallelogram, since the other gives a parallelogram.
					Therefore, we have $\deg(f)=\deg(g)=1$,
					and hence $\mu_R=\mu_L=1$.
		\item If the motion $\C$ is parallel,
					then $\mu_L=0, \mu_R=0$ and $\mu_O>0$.
					Since one of the intersection points lies on the antiparallel component, we have $\deg(f)=\deg(g)=1$ and thus $\mu_O=1$.
		\item Similarly for a rhombus as one of the intersection points lies on a degenerate component.
	  \item If $\C$ is general, then $\deg(f)=\deg(g)=2$
					since the coordinates of the missing vertex both lie on the curve $\C$.
					Since all possible NAC-colorings are active, we know that
					$\mu_O, \mu_L$ and $\mu_R$ are non-zero. Hence, $\mu_O=\mu_L=\mu_R=1$.
		\item If $\C$ is an odd deltoid, then $\mu_L>0, \mu_O>0$ and $\mu_R=0$
					(see Table~\ref{tab:colQuadrilateral}).
					Since the degree of $g$ is two, $\mu_O=\mu_L=1$.
					The equations imply $\deg(f)=1$ which is indeed the case,
					since one point of the intersection does not lie on $\C$,
					but on the degenerate component.
					Similarly, $\mu_O=\mu_R=1$ for an even deltoid.		
	\end{itemize}
	
	If $\C$ is the degenerate component of an odd deltoid,
	then $\mu_L=0$, $\mu_O=0$ and $\mu_R>0$.
	From the first equation in \eqref{eq:ramFromForQlats}, 
	$\mu_R=\mu_{f}\cdot \deg(f) = \deg(f)$, which
	is one since $v_3$ coincides with $v_1$.
	Therefore, $\mu_R=1$. 
	Other degenerate motions follow analogously.
\end{proof}

The main step of our classification approach is to find possible sets of active NAC-colorings of an algebraic motion
using the condition on $\mu$-numbers from Theorem~\ref{thm:ramificationFormula}
for restrictions to 4-cycles.
We introduce a consistency condition without assuming an algebraic motion. 

\begin{definition}	\label{def:consistentWithFourCycles}
	Let $H_1,\dots, H_k$ be some 4-cycles of a graph $G$ such that for each $H_i$
	there exists $\delta\in\nac{G}$ with $\delta(E_{H_i}) = \{\red{}, \blue{}\}$.
	A subset~$N\subset\nac{G}$
	is called \emph{consistent with the 4-cycles} $H_1,\dots, H_k$
	if there exists a vector of nonnegative integers $(\mu_\delta)_{\delta\in\nac{G}}$ such that:
	\begin{enumerate}
		\item\label{enumi:zeroMu} $\mu_\delta\neq 0$  if and only if $\delta\in N$, and
		\item\label{enumi:motionTypeEqForMu} for each 4-cycle $H_i$, there exists a positive integer $d_i$ such that
			\begin{equation*}
				\sum_{\mathclap{\substack{\delta\in\nac{G} \\ \delta|_{E_{H_i}}=\, \delta'}}} \mu_\delta = d_i
			\end{equation*}
			for all $\delta' \in \{\delta|_{E_{H_i}}\in \nac{H_i} \colon \delta \in N\}$.
	\end{enumerate}
	We omit the specification of the 4-cycles if we consider all 4-cycles satisfying the assumption.
\end{definition}

The following corollary justifies that the sets of NAC-colorings consistent with 4-cycles
form a superset of sets of active NAC-colorings of algebraic motions.

\begin{corollary}	\label{cor:mustBeConsistent}
	Let $\C$ be an algebraic motion of a graph $G$.
	Let $H_	1,\dots, H_k$ be some 4-cycles of $G$.
	Let $\C_1,\dots,\C_k$ be the images of $\C$ by the projections
	to the realizations of $H_1,\dots, H_k$.
	If $\C_1,\dots,\C_k$ are curves,
	then the set of active NAC-colorings $\nacC{G}{\C}$ is consistent with the 4-cycles $H_1,\dots, H_k$,
\end{corollary}
\begin{proof}
	We set $\mu_\delta:=\mu(\delta,\C)$. 
	The condition~\ref{enumi:zeroMu} of Definition~\ref{def:consistentWithFourCycles} follows
	from Remark~\ref{rem:mu_active}.
	Regarding~\ref{enumi:motionTypeEqForMu}: for $i\in\{1,\dots,k\}$,
	if $\delta'\in\{\delta|_{E_{H_i}}\in\nac{H_i}\colon \delta \in \nacC{G}{\C}\}\myeq{\ref{cor:muOfExtIsAlsoZero}}\nacC{H_i}{\C_i}$,
	then
	\begin{equation*}
		\sum_{\mathclap{\substack{\delta\in\nac{G} \\ \delta|_{E_{H_i}}=\,\delta'}}} \mu_\delta \myeq{\ref{thm:ramificationFormula}}\mu(\delta',\C_i)\cdot\deg(f_i) \myeq{\ref{thm:muQlatsIsOne}} \deg(f_i)>0\,.\qedhere
	\end{equation*}
\end{proof}

\section{Consistent motion types}
\label{sec:consistency}
Corollary~\ref{cor:mustBeConsistent} justifies that
if there is an algebraic motion of $(G, \lambda)$,
then the set of active NAC-colorings is consistent with 4-cycles.
In order to classify proper flexible labelings of a graph $G$,
we want to determine all possible candidate sets of active NAC-colorings.
A possible approach would be to check consistency with all 4-cycles for each subset of~$\nac{G}$.

We use a slightly modified approach, which seems to be more efficient:
instead of assuming a set of NAC-colorings,
we make an assumption on the motions of quadrilaterals.
The motions determine the active NAC-colorings of the 4-cycles.
From these we construct the system of equations according to Definition~\ref{def:consistentWithFourCycles} and check possible solutions.
The nonzero entries of a solution give us a set of active
NAC-colorings consistent with 4-cycles. 
We also say ``motion types consistent with 4-cycles'',
since a motion type defines the equation \ref{enumi:motionTypeEqForMu} in Definition~\ref{def:consistentWithFourCycles}.
If a combination of motion types is not consistent with a subset of 4-cycles,
then every extension of this combination to all 4-cycles can be skipped.
Hence, we iteratively extend the assumption on motion types and check the consistency. 
We describe the procedure in detail now.

\consistencycheck{} (Algorithm~\ref{alg:checkConsistent}): 
We check consistency for given types of motion of 4-cycles.
Let $H_1,\dots, H_k$ be 4-cycles of $G$ such that
there exists a $\delta\in\nac{G}$ with $\delta(E_i) = \{\red{}, \blue{}\}$.
Since we are interested only in proper flexible labelings,
the types of active NAC-colorings $m_i$ of a 4-cycle $H_i$ must correspond to
a non-degenerate motion (compare with Table~\ref{tab:colQuadrilateral}):
\begin{itemize}
  \item $\caseP=\left\{\colO\right\}=\{O\}$ (parallelogram/rhombus),
  \item $\caseA=\left\{\colL, \colR\right\}=\{L, R\}$ (antiparallelogram),
  \item $\caseO=\left\{\colO, \colL\right\}=\{O, L\}$ (odd deltoid),
  \item $\caseE=\left\{\colO, \colR\right\}=\{O, R\}$ (even deltoid), or
  \item $\caseG=\left\{\colO, \colL, \colR\right\}=\{O, L, R\}$ (general).
\end{itemize}
We use the notation $\caseP, \caseA, \caseO, \caseE, \caseG$ both for sets of types of active NAC-colorings and motion types.
Having an ansatz $(m_1,\dots, m_k) \in \{\caseP, \caseA, \caseO, \caseE, \caseG\}^k$ 
of motions types of the 4-cycles,
we construct the set of equations in variables $\mu_\delta$ for all $\delta\in\nac{G}$.
According to Corollary~\ref{cor:muOfExtIsAlsoZero},
$\mu_\delta=0$ for all $\delta\in\nac{G}$ such that the type of $\delta|_{E_{H_i}}$
is not in $m_i$ for some~$H_i$.
For each~$H_i$, the equations from Definition~\ref{def:consistentWithFourCycles}\ref{enumi:motionTypeEqForMu}
are added for $\delta$ whose restriction has a type in $m_i$.
Since we do not know the $d_i$, we consider the equations with $d_i$ eliminated.
In particular, there is no equation added if $m_i=\caseP$.
If the obtained system of equations allows a nontrivial nonnegative solution,
then $(m_1,\dots, m_k)$ is consistent with 4-cycles.
Example~\ref{ex:K33} shows the system of equations for a specific assumption of motion types
of 4-cycles of $\Ktt$. 

We remark that if $(m_1,\dots, m_k)$ are motion types of an algebraic motion
with a flexible labeling $\lambda$, then all $m_i\in \{\caseP, \caseA, \caseE, \caseO\}$
enforce some edge lengths to be equal, e.g.\ if $m_i\in\{\caseA, \caseP\}$, 
then the opposite edges of $H_i$ have the same lengths.

\begin{algorithm}[ht]
	\caption{\textsc{CheckConsistency}}
	\label{alg:checkConsistent}
	\begin{algorithmic}[1]
		\Require Motion types $(m_1,\dots, m_k) \in \{\caseG, \caseP, \caseA, \caseE, \caseO\}^k$
			for 4-cycle subgraphs $H_1,\dots, H_k$
		\Ensure Consistency of the motion types with the 4-cycles
		\State $R:=\emptyset$
		\For{$\delta\in \nac{G}$}
			\If{$\exists i \in \{1,\dots,k\}: \text{type of } \delta|_{E_{H_i}} \notin m_i$}
				\State Add the equation $\mu_\delta=0$ into $R$.
			\EndIf
		\EndFor
		\For{$i \in \{1,\dots,k\}$}
			\State Add the equations
			\begin{equation*}
				\sum_{\mathclap{\substack{\delta\in\nac{G} \\ \text{type of }\delta|_{E_{i}}=\,t_1}}} \mu_\delta = \dots = \sum_{\mathclap{\substack{\delta\in\nac{G} \\ \text{type of }\delta|_{E_{i}}=\,t_{s_i}}}} \mu_\delta
			\end{equation*}
			\Statex\hspace{\algorithmicindent}into $R$, where $\{t_1, \dots, t_{s_i}\}=m_i$. \label{alg:line:sums}
		\EndFor
		\State\Return Boolean value whether $R$ has a nonzero solution with nonnegative entries 

	\end{algorithmic}
\end{algorithm}

\consistenttypes{} (Algorithm~\ref{alg:findConsistent}): In order to find all combinations of consistent motion types,
let $H_1, \dots, H_\ell$ be all nontrivially colored 4-cycles of $G$.
In the $i$-th iteration, we extend all consistent motion types for $H_1, \dots, H_{i-1}$ by
all $m_i \in \{\caseG, \caseP, \caseA, \caseE, \caseO\}$
and do \consistencycheck.

\begin{remark}	
	\label{rem:motion_types_necessary_cond}
	Moreover, we check also the following necessary conditions:
	\begin{enumerate}
	  	\item\label{enumi:general} Let $\lambda$ be edge lengths enforced by $(m_1,\dots, m_k)$.
		  	The edge lengths of the 4-cycles with the motion type $\caseG$ must be general,
		  	i.e., pairs of edges cannot have the same length.
	  	\item\label{enumi:K23} For $\lambda$ as in \ref{enumi:general}, the subgraphs isomorphic to $K_{2,3}$
	  		cannot have all edge lengths to be the same. Otherwise, two vertices have to coincide.
	  	\item\label{enumi:forbiddenMotionTypes} If there are two 4-cycles $(v_1, v_2, v_3, v_4)$ and $(v_1, v_2, v_5, v_4)$,
	  		then they are not allowed to have motion types $\caseA\caseA$, $\caseP\caseP$,
	  		$\caseE\caseO$ or $\caseA\caseE$. The first two cases would coincide vertices,
	  		the third one would imply that one of the deltoids is actually a rhombus,
	  		and the last one would force
	  		the antiparallelogram to have all lengths equal. 
	\end{enumerate}
\end{remark}
A possible heuristic to reduce
the number of cases is to sort the 4-cycles so that $H_{i+1}$
shares some edges with $H_i$, if possible.

Having the set of possible motion types from \consistenttypes,
we consider only non-isomorphic cases.
Motion types  $(m_1,\dots, m_\ell)$ and $(m'_1,\dots, m'_\ell)$
are isomorphic  if there exists an element of the automorphisms group of $G$
such that the sets of edges with the same lengths enforced by $(m_1,\dots, m_\ell)$
are the images of the sets of edges with the same lengths enforced by $(m'_1,\dots, m'_\ell)$.
In other words, motion types of 4-cycles are preserved when the automorphism is applied,
taking into account that $\caseO$ and $\caseE$ sometimes interchange.

For a representative of each isomorphism class,
we consider the system of equations from \consistencycheck. 
For each nonnegative solution of the system,
a set of possible active NAC-colorings consistent with 4-cycles is given 
by Definition~\ref{def:consistentWithFourCycles}\ref{enumi:zeroMu}.
Clearly, the solutions with the same nonzero entries give the same 
sets of active NAC-colorings. 

\begin{algorithm}[htb]
	\caption{\textsc{ConsistentTypes}}
	\label{alg:findConsistent}
	\begin{algorithmic}[1]
		\Require All 4-cycle subgraphs $H_1,\dots, H_\ell$
		\Ensure Set $S$ of possible motion types of 4-cycles.
		\State $S_{\text{prev}}:=\emptyset$
		\For{$i \in \{1, \dots \ell\}$}
			\State $S:=\emptyset$
			\For{$(m_1, \dots, m_{i-1}) \in S_{\text{prev}}$}
				\For{$m_i \in \{\caseG, \caseP, \caseA, \caseE, \caseO \} $}
					\If{\textsc{CheckConsistency}($(m_1, \dots, m_{i})$,$H_1, \dots, H_i$) 
					  \textbf{and} \ref{rem:motion_types_necessary_cond}\ref{enumi:general}--\ref{enumi:forbiddenMotionTypes}}
						\State Add $(m_1, \dots, m_{i})$ to $S$
					\EndIf
				\EndFor
			\State $S_{\text{prev}}:=S$
			\EndFor
		\EndFor
		\State \Return $S$
	\end{algorithmic}
\end{algorithm}

\begin{remark}
For each set of active NAC-colorings for motion types $M=(m_1,\dots, m_\ell)$,
we determine all 4-cycles with orthogonal diagonals due to being deltoids,
an active NAC-coloring 
satisfying the assumption of  Proposition~\ref{prop:orthogonalDiagonals}, or a consequence of the previous two.
	\label{rem:orthogonal_diags}
	Assume now that $H_i$ has orthogonal diagonals. The following cases are checked:
	\begin{enumerate}
	  \item If $m_i=\caseA$, then $M$ is invalid, since antiparallelograms cannot have orthogonal diagonals.
	  \item If $m_i=\caseP$, then all edges of $H_i$ have the same lengths. 
	  	The conditions~\ref{enumi:general} and~\ref{enumi:K23} in Remark~\ref{rem:motion_types_necessary_cond}
	  	have to be verified for edge lengths enforced by $M$ and this extra condition.
	  \item If $m_i=\caseG$ and two incident edges of $H_i$ are enforced by $M$ to have the same lengths,
	  	then also the lengths of the other two edges are equal, which contradicts that $H_i$ is general. 
	\end{enumerate}
\end{remark}
The necessary conditions in Remarks~\ref{rem:motion_types_necessary_cond} and~\ref{rem:orthogonal_diags}
are sufficient for $K_{3,3}$, see Example~\ref{ex:K33}, but not in general --- 
there are motion types of graph $Q_1$ satisfying the conditions but yielding coinciding vertices,
see \ref{Q1:g} in Table~\ref{tab:Q1_consistent_motion_types} in Section~\ref{sec:Q1classification}.

Here is a summary on how to get pairs of motion types and a set of active NAC-colorings
that are possible for non-degenerate motions of the graph, modulo graph automorphism:
\begin{enumerate}
  \item Find all motion types consistent with 4-cycles using \consistenttypes\ and Remark~\ref{rem:motion_types_necessary_cond}.
  \item Identify isomorphism  class of the consistent motion types.
  \item For a representative of each class, determine all possible sets of active NAC-colorings.
  \item Check the conditions of Remark~\ref{rem:orthogonal_diags} related to orthogonality of diagonals.
\end{enumerate}
We provide our implementation of the method in \cite{flexrilog}.

\begin{example}
\label{ex:K33}
We illustrate our method on the well know example of $\Ktt$ by describing all proper flexible labelings.
Dixon \cite{Dixon} showed two ways to construct such a labeling, see below.
Walter and Husty \cite{WalterHusty} proved that these are the only two cases.
We prove the same using our method.
Let $V_{\Ktt}=\{1,\dots,6\}$ and 
$E_{\Ktt}=\{ij \colon i\in\{1,3,5\}, j\in\{2,4,6\}\}$, see Figure~\ref{fig:K33}.
A proper flexible labeling $\lambda$ of $\Ktt$ with an algebraic motion $\C$ is called
\begin{description}
	\item[{Dixon~I type}] if for every realization $\rho\in\C$,
		the points $\rho(1)$, $\rho(3)$ and $\rho(5)$, resp.\ $\rho(2)$, $\rho(4)$ and $\rho(6)$
		are collinear and these two lines are orthogonal, or
	\item[{Dixon~II type}] if for every realization $\rho\in\C$,
		the points $\rho(1)$, $\rho(3)$ and $\rho(5)$,	resp.\ $\rho(2)$, $\rho(4)$ and $\rho(6)$,
		lie in vertices of a rectangle $R_1$, resp.\ $R_2$,
		such that the intersection of diagonals of $R_1$ and $R_2$
		is the same and the sides of $R_1$ are parallel or perpendicular to the sides of $R_2$
		(see Figure~\ref{fig:DixonII}).
\end{description}

There are only two NAC-colorings of $\Ktt$ up to symmetries and conjugation,
see Figure~\ref{fig:K33}.
We denote them by $\omega_k$ for $k\in V_{\Ktt}$ and $\epsilon_{kl}$ for $kl\in E_\Ktt$,
where for an edge $ij\in E_{\Ktt}$:
\begin{align*}
	\omega_k(ij)=\blue{} &\iff k \in\{i,j\}\,, \\
	\epsilon_{kl}(ij)=\blue{} &\iff \{k,l\}=\{i,j\} \lor  \{k,l\}\cap\{i,j\}=\emptyset\,.
\end{align*}

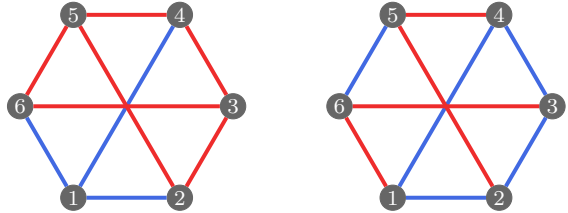
\begin{figure}[htb]
	\centering	
		\begin{tikzpicture}[scale=1.4]
			\node[lnode] (a1) at (-0.5, -0.866025) {1};
			\node[lnode] (a2) at (0.5, -0.866025) {2};
			\node[lnode] (a3) at (1., 0.) {3};
			\node[lnode] (a4) at (0.5, 0.866025) {4};
			\node[lnode] (a5) at (-0.5, 0.866025) {5};
			\node[lnode] (a6) at (-1.,  0.) {6};
			\draw[bedge] (a1)edge(a2) (a1)edge(a4) (a1)edge(a6);
			\draw[redge] (a2)edge(a3) (a2)edge(a5) (a3)edge(a4) (a3)edge(a6) (a5)edge(a4) (a5)edge(a6);
		
			\begin{scope}[xshift=3cm]
			\node[lnode] (a1) at (-0.5, -0.866025) {1};
			\node[lnode] (a2) at (0.5, -0.866025) {2};
			\node[lnode] (a3) at (1., 0.) {3};
			\node[lnode] (a4) at (0.5, 0.866025) {4};
			\node[lnode] (a5) at (-0.5, 0.866025) {5};
			\node[lnode] (a6) at (-1.,  0.) {6};
				\draw[bedge] (a1)edge(a2) (a1)edge(a4) (a2)edge(a3) (a3)edge(a4) (a5)edge(a6);
				\draw[redge] (a1)edge(a6) (a2)edge(a5) (a3)edge(a6) (a5)edge(a4) ;
			\end{scope}
		\end{tikzpicture}
	\caption{$\Ktt$ with NAC-colorings $\omega_1$ and $\epsilon_{56}$.}
	\label{fig:K33}
\end{figure}

Let $H_1, \dots, H_9$ be the 4-cycles $(1, 2, 3, 4)$, $(1, 2, 3, 6)$, $(1, 4, 3, 6)$, 
$(1, 2, 5, 6)$, $(2, 3, 6, 5)$, $(2, 3, 4, 5)$, $(3, 4, 5, 6)$, $(1, 4, 5, 6)$, $(1, 2, 5, 4)$ respectively.
To illustrate \consistencycheck, 
we assume motion types of some 4-cycles and construct the equations:
\begin{align*}
	H_1 &: \caseE &
	\mu_{\omega_{3}} =
	\mu_{\omega_{1}} =
	\mu_{\epsilon_{36}} &=
	\mu_{\epsilon_{16}} = 0\,, \\ &&
	\textcolor{colGray}{\mu_{\epsilon_{34}}} + \textcolor{colGray}{\mu_{\epsilon_{14}}} + \textcolor{colGray}{\mu_{\epsilon_{23}}} + \mu_{\epsilon_{12}} &=
	\textcolor{colGray}{\mu_{\omega_{4}}} + \mu_{\epsilon_{45}} + \textcolor{colGray}{\mu_{\omega_{2}}} + \textcolor{colGray}{\mu_{\epsilon_{25}}}\,,
	\\
	H_2 &: \caseP &
	\mu_{\omega_{3}} =
	\mu_{\omega_{1}} =
	\mu_{\epsilon_{34}} &=
	\mu_{\epsilon_{14}} =
	\mu_{\omega_{6}} =
	\mu_{\epsilon_{56}} =
	\mu_{\omega_{2}} =
	\mu_{\epsilon_{25}} = 0\,,
	\\
	H_3 &: \caseA &
	\mu_{\epsilon_{36}} =
	\mu_{\epsilon_{16}} =
	\mu_{\epsilon_{34}} &=
	\mu_{\epsilon_{14}} =0\,, \\ &&
	\textcolor{colGray}{\mu_{\omega_{3}}} + \textcolor{colGray}{\mu_{\omega_{1}}} + \textcolor{colGray}{\mu_{\epsilon_{23}}} + \mu_{\epsilon_{12}} &=
	\textcolor{colGray}{\mu_{\omega_{6}}} + \textcolor{colGray}{\mu_{\epsilon_{56}}} + \textcolor{colGray}{\mu_{\omega_{4}}} + \mu_{\epsilon_{45}}\,,
	\\
	H_4 &: \caseG &
	\textcolor{colGray}{\mu_{\omega_{5}}} + \textcolor{colGray}{\mu_{\omega_{1}}} + \mu_{\epsilon_{45}} + \textcolor{colGray}{\mu_{\epsilon_{14}}} &=
	\textcolor{colGray}{\mu_{\epsilon_{56}}} + \textcolor{colGray}{\mu_{\epsilon_{16}}} + \textcolor{colGray}{\mu_{\epsilon_{25}}} + \mu_{\epsilon_{12}}=
	\textcolor{colGray}{\mu_{\omega_{6}}} + \textcolor{colGray}{\mu_{\epsilon_{36}}} + \textcolor{colGray}{\mu_{\omega_{2}}} + \textcolor{colGray}{\mu_{\epsilon_{23}}}\,,
	\\
	H_6 &: \caseO &
	\mu_{\omega_{5}} =
	\mu_{\omega_{3}} =
	\mu_{\epsilon_{56}} &=
	\mu_{\epsilon_{36}} =0\,, \\ &&
	\textcolor{colGray}{\mu_{\omega_{4}}} + \textcolor{colGray}{\mu_{\epsilon_{14}}} + \textcolor{colGray}{\mu_{\omega_{2}}} + \mu_{\epsilon_{12}} &=
	\mu_{\epsilon_{45}} + \textcolor{colGray}{\mu_{\epsilon_{34}}} + \textcolor{colGray}{\mu_{\epsilon_{25}}} + \textcolor{colGray}{\mu_{\epsilon_{23}}}\,,
	\\
	H_9 &: \caseP &
	\mu_{\omega_{5}} =
	\mu_{\omega_{1}} =
	\mu_{\epsilon_{56}} &=
	\mu_{\epsilon_{16}} =
	\mu_{\omega_{4}} =
	\mu_{\epsilon_{34}} =
	\mu_{\omega_{2}} =
	\mu_{\epsilon_{23}} =0\,.
\end{align*}
The summands that are immediately zero are colored gray.
There is no nontrivial solution, since
$\mu_{\epsilon_{45}}$ and~$\mu_{\epsilon_{12}}$ must be zero.
Hence, these motion types are not consistent with 4-cycles.

The output of the whole method described in Section~\ref{sec:ramification}
obtained by computer is summarized in Table~\ref{tab:kttConsistentColorings} (see~{\cite{flexrilog,LegerskySupportingMaterial}}).
The motion types consistent with $H_1, \dots, H_9$ and the corresponding sets 
of active NAC-colorings are in the first and second column.
The third column indicates the number of motion types isomorphic to the chosen representative.

\begin{table}[htb]
	\centering
		\begin{tabular}{*{8}{C@{\hspace{7pt}}}Cccc}
			\multicolumn{9}{c}{\textbf{Motions types}} & \textbf{Active} & \textbf{\#} & \textbf{Type} \\
			\multicolumn{9}{c}{\textbf{ of} $(H_1,\dots,H_9)$} & \textbf{NAC-colorings} & \textbf{isomorphic} & \textbf{of motion} \\ \midrule
			\caseG{} & \caseG{} & \caseG{} & \caseG{} & \caseG{} & \caseG{} & \caseG{} & \caseG{} & \caseG{}
				& $\nac{\Ktt}$ & 1 & \multirow{3}{*}{Dixon I}\\
			\caseO{} & \caseO{} & \caseO{} & \caseG{} & \caseG{} & \caseG{} & \caseG{} & \caseG{} & \caseG{}
				& $\left\{\epsilon_{12}, \epsilon_{23}, \epsilon_{34}, \epsilon_{14}, \epsilon_{16}, \epsilon_{36}, \omega_{1}, \omega_{3}\right\}$
				& 6 \\
			\caseP{} & \caseO{} & \caseO{} & \caseG{} & \caseG{} & \caseO{} & \caseG{} & \caseG{} & \caseE{}
				& $\left\{\epsilon_{12}, \epsilon_{23}, \epsilon_{34}, \epsilon_{14}\right\}$ 
				& 9 \\ \midrule
			\caseP{} & \caseG{} & \caseG{} & \caseG{} & \caseA{} & \caseG{} & \caseG{} & \caseA{} & \caseG{}
				& $\left\{\epsilon_{12}, \epsilon_{34}, \omega_{5}, \omega_{6}\right\}$ 
				& 18 & Dixon II
		\end{tabular}
	\caption{Consistent motion types and active NAC-colorings of $\Ktt$.}
	\label{tab:kttConsistentColorings}
\end{table}

We show that the first three lines are of Dixon~I type.
Since $\epsilon_{12},\epsilon_{23},\epsilon_{34}$ are active,
we have	that 
the 4-cycles $(3, 4, 5, 6)$, $(1, 4, 5, 6)$ and $(1, 2, 5, 6)$
have perpendicular diagonals by Proposition~\ref{prop:orthogonalDiagonals}.
It immediately follows that $1,3$ and $5$, resp.\ $2, 4$ and $6$
are collinear and these two lines are orthogonal.
We remark that the second and third line are special cases when some 4-cycles are deltoids.
The 4-cycle $H_1$ in the third line is actually a rhombus. They are depicted in Figure~\ref{fig:DixonII}.

To show that the last line is a Dixon~II motion,
consider a realization compatible with edge lengths 
enforced by the motion types.
The positions of the vertices of the antiparallelogram $(2,3,6,5)$
together with the edge lengths
determine the position of the vertex $4$, since the 4-cycle $(3, 4, 5, 6)$
has perpendicular diagonals due to the active NAC-coloring $\epsilon_{12}$.
The position of $1$ is then determined since $(1,2,3,4)$ is a parallelogram.
Since $(1,4,5,6)$ is an antiparallelogram, the vertices
lie on the two rectangles as required, see Figure~\ref{fig:DixonII}.

\end{example}
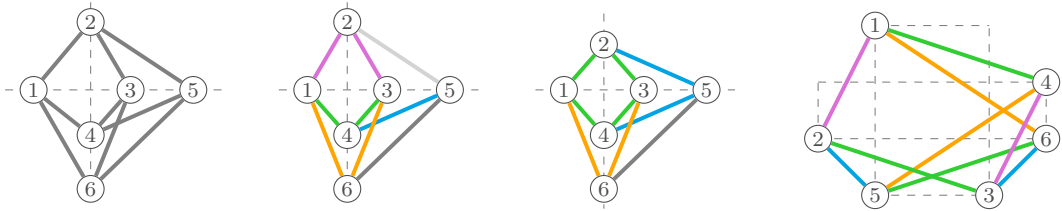
\begin{figure}[htb]
	\centering
	\begin{tikzpicture}[scale=0.75]
			\tikzset{vertex/.style={lnodeR}}
			\draw[gridl, dashed] (-1.5,0)edge(2.3,0);
			\draw[gridl, dashed] (0,1.55)edge(0,-2.1);
			\node[vertex] (2) at (1.8, 0) {5};
			\node[vertex] (5) at (-1., 0) {1};
			\node[vertex] (7) at (0.7, 0) {3};
			\node[vertex] (1) at (0, -1.75) {6};
			\node[vertex] (6) at (0,  1.2) {2};
			\node[vertex] (4) at (0, -0.8) {4};
			\draw[edge]  (6)edge(5) (5)edge(4) (7)edge(4) (7)edge(6);
			\draw[edge] (2)edge(4) (2)edge(6);
			\draw[edge] (1)edge(5) (7)edge(1) ;
			\draw[edge] (2)edge(1);
			\begin{scope}[xshift=4.5cm]
				\draw[gridl, dashed] (-1.2,0)edge(2.3,0);
				\draw[gridl, dashed] (0,1.55)edge(0,-2.1);
				\node[vertex] (2) at (1.8, 0) {5};
				\node[vertex] (5) at (-0.7, 0) {1};
				\node[vertex] (7) at (0.7, 0) {3};
				\node[vertex] (1) at (0, -1.75) {6};
				\node[vertex] (6) at (0,  1.2) {2};
				\node[vertex] (4) at (0, -0.8) {4};
				\draw[edge, col1] (7)edge(4)  (5)edge(4) ;
				\draw[edge, col2]  (7)edge(6) (6)edge(5) ;
				\draw[edge, col4] (2)edge(4) ;
				\draw[edge, col6] (2)edge(6);
				\draw[edge, col3] (1)edge(5) (7)edge(1) ;
				\draw[edge] (2)edge(1);
			\end{scope}
			\begin{scope}[xshift=9cm]
				\draw[gridl, dashed] (-1.2,0)edge(2.3,0);
				\draw[gridl, dashed] (0,1.35)edge(0,-2.1);
				\node[vertex] (2) at (1.8, 0) {5};
				\node[vertex] (5) at (-0.7, 0) {1};
				\node[vertex] (7) at (0.7, 0) {3};
				\node[vertex] (1) at (0, -1.75) {6};
				\node[vertex] (6) at (0,  0.8) {2};
				\node[vertex] (4) at (0, -0.8) {4};
				\draw[edge, col1]  (6)edge(5) (5)edge(4) (7)edge(4) (7)edge(6);
				\draw[edge, col4] (2)edge(4) (2)edge(6);
				\draw[edge, col3] (1)edge(5) (7)edge(1) ;
				\draw[edge] (2)edge(1);
			\end{scope}
		\end{tikzpicture}
		\qquad
		\begin{tikzpicture}[scale=0.75]
			\tikzset{lnode/.style={lnodeR}}
			\draw[gridl, dashed] (1,0)edge(3,0) (1,3)edge(3,3) (0,1)edge(4,1) (0,2)edge(4,2);
			\draw[gridl, dashed] (0,1)edge(0,2) (4,1)edge(4,2) (1,0)edge(1,3) (3,0)edge(3,3); 
			\node[lnode] (7) at (1,3) {1};
			\node[lnode] (5) at (4,1) {6};
			\node[lnode] (3) at (3,0) {3};
			\node[lnode] (6) at (0,1) {2};
			\node[lnode] (2) at (1,0) {5};
			\node[lnode] (4) at (4,2) {4};
			\draw[edge, col4] (5)edge(3) (2)edge(6) ;
			\draw[edge, col3]  (2)edge(4)  (5)edge(7);
			\draw[edge, col1] (2)edge(5) (7)edge(4) (3)edge(6);
			\draw[edge, col2] (7)edge(6) (3)edge(4) ;
		\end{tikzpicture}
  \caption{The first three realizations illustrate Dixon I motion of $\Ktt$ --- a general one and two special cases.
  The right one is Dixon II motion (same colors indicate same edge lengths).}
  \label{fig:DixonII}
\end{figure}
We remark that a computer-free proof can be obtained by combining the results from Sections~\ref{sec:comparingLC} and~\ref{sec:ramification} with the approach
used in~\cite{Gallet2019} to classify all motions of $K_{3,3}$ on the sphere.

\section{Classification of motions of \texorpdfstring{$Q_1$}{Q1}}
\label{sec:Q1classification}
The goal of this section is to classify all proper flexible labelings of the graph $Q_1$
using the tools developed in the previous sections.
We determine consistent motion types and active NAC-colorings by the method from Section~\ref{sec:consistency}.
These are obtained by computer using our implementation of the method in the \flexrilog\ package~\cite{flexrilog}.
Using Proposition~\ref{prop:degenerateTriangle}, it appears that the unique triangle is actually degenerate for every consistent motion type
(a proof without the method from the previous section was presented in~\cite{GLSeuroCG}).
Then we obtain necessary algebraic conditions on $\lambda$ 
from singleton NAC-colorings by the technique from Section~\ref{sec:comparingLC}.
Using the computer algebra system \sage{}~\cite{sagemath}, we identify six groups of motion types of 4-cycles,
giving altogether eight motion families --- irreducible algebraic sets of proper flexible labelings (see~\cite{LegerskySupportingMaterial} for the computations and \cite{ClassificationQ1} for a detailed analysis). 
Animations of these motions can be found in \cite{LegerskyAnimations}.
For the general idea we present the analysis of two cases in the next sections.
The remaining ones are illustrated in Figure~\ref{fig:Q1remaining} and fully analyzed in \cite{ClassificationQ1}.
The NAC-colorings of~$Q_1$ are shown in Figure~\ref{fig:Q1NACs}.
The figure also depicts the vertex labels we use. 

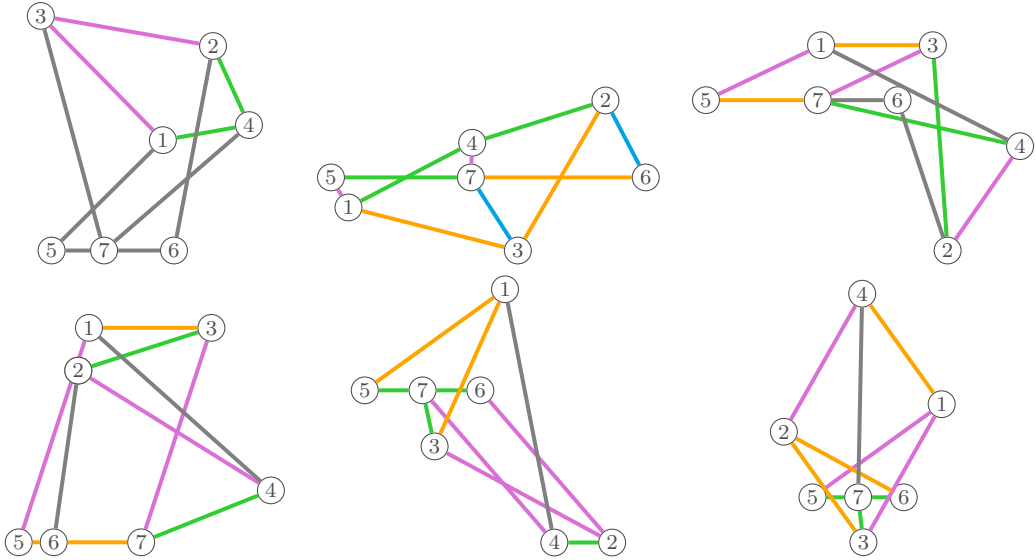
\begin{figure}[ht]
	\centering
	\begin{tabular}{ccc}
		\begin{tikzpicture}[scale=0.23]
			\tikzset{lnodesmall/.style={lnodeR}}
			\node[lnodesmall] (1) at (6.3640, 6.3640) {1};
			\node[lnodesmall] (2) at (9.2324, 11.791) {2};
			\node[lnodesmall] (3) at (-0.61602, 13.525) {3};
			\node[lnodesmall] (4) at (11.288, 7.2326) {4};
			\node[lnodesmall] (5) at (0.00000, 0.00000) {5};
			\node[lnodesmall] (6) at (7.0000, 0.00000) {6};
			\node[lnodesmall] (7) at (3.0000, 0.00000) {7};
				\draw[edge, col2](1)edge(3) (2)edge(3);
				\draw[edge, col1](1)edge(4) (2)edge(4);
				\draw[edge](4)edge(7) (2)edge(6) (7)edge(6) (5)edge(7) (1)edge(5) (3)edge(7);
		\end{tikzpicture}
		&
		\begin{tikzpicture}[scale=0.23]
			\tikzset{lnodesmall/.style={lnodeR}}
			\node[lnodesmall] (1) at (1.0000, -1.7320) {1};
			\node[lnodesmall] (2) at (15.696, 4.4375) {2};
			\node[lnodesmall] (3) at (10.686, -4.2171) {3};
			\node[lnodesmall] (4) at (8.0769, 1.9985) {4};
			\node[lnodesmall] (5) at (0.00000, 0.00000) {5};
			\node[lnodesmall] (6) at (18.000, 0.00000) {6};
			\node[lnodesmall] (7) at (8.0000, 0.00000) {7};
				\draw[edge, col1](5)edge(7) (1)edge(4) (2)edge(4);
				\draw[edge, col2](4)edge(7) (1)edge(5);
				\draw[edge, col3](6)edge(7) (1)edge(3) (2)edge(3);
				\draw[edge, col4](2)edge(6) (3)edge(7);
		\end{tikzpicture}
		&
		\begin{tikzpicture}[scale=0.21]
			\tikzset{lnodesmall/.style={lnodeR}}
			\node[lnodesmall] (1) at (7.2077, 3.4711) {1};
			\node[lnodesmall] (2) at (15.135, -9.4958) {2};
			\node[lnodesmall] (3) at (14.208, 3.4711) {3};
			\node[lnodesmall] (4) at (19.671, -2.9059) {4};
			\node[lnodesmall] (5) at (0.00000, 0.00000) {5};
			\node[lnodesmall] (6) at (12.000, 0.00000) {6};
			\node[lnodesmall] (7) at (7.0000, 0.00000) {7};
				\draw[edge, col1](4)edge(7) (2)edge(3);
				\draw[edge, col2](1)edge(5) (3)edge(7) (2)edge(4);
				\draw[edge, col3](5)edge(7) (1)edge(3);
				\draw[edge](2)edge(6) (6)edge(7) (1)edge(4);
		\end{tikzpicture}
		\\
		\begin{tikzpicture}[scale=0.23]
			\tikzset{lnodesmall/.style={lnodeR}}
			\begin{scope}
				\node[lnodesmall] (1) at (4.0172, 12.364) {1};
				\node[lnodesmall] (2) at (3.4058, 9.9007) {2};
				\node[lnodesmall] (3) at (11.017, 12.364) {3};
				\node[lnodesmall] (4) at (14.419, 2.9933) {4};
				\node[lnodesmall] (5) at (0.00000, 0.00000) {5};
				\node[lnodesmall] (6) at (2.0000, 0.00000) {6};
				\node[lnodesmall] (7) at (7.0000, 0.00000) {7};
				\draw[edge, col1](4)edge(7) (2)edge(3);
				\draw[edge, col2](1)edge(5) (3)edge(7) (2)edge(4);
				\draw[edge, col3](5)edge(6) (6)edge(7) (1)edge(3);
				\draw[edge](2)edge(6) (1)edge(4);
				\node[lnodesmall] at (2) {2};
			\end{scope}
		\end{tikzpicture}
		&
		\begin{tikzpicture}[scale=0.38]
			\tikzset{lnodesmall/.style={lnodeR}}
			\node[lnodesmall] (1) at (4.8541, 3.5267) {1};
			\node[lnodesmall] (2) at (8.5612, -5.3099) {2};
			\node[lnodesmall] (3) at (2.4170, -1.9560) {3};
			\node[lnodesmall] (4) at (6.5612, -5.3099) {4};
			\node[lnodesmall] (5) at (0.00000, 0.00000) {5};
			\node[lnodesmall] (6) at (4.0000, 0.00000) {6};
			\node[lnodesmall] (7) at (2.0000, 0.00000) {7};
				\draw[edge, col1](6)edge(7) (5)edge(7) (3)edge(7) (2)edge(4);
				\draw[edge, col2](4)edge(7) (2)edge(6) (2)edge(3);
				\draw[edge, col3](1)edge(5) (1)edge(3);
				\draw[edge](1)edge(4);
		\end{tikzpicture}
		&
		\begin{tikzpicture}[scale=0.3]
			\tikzset{lnodesmall/.style={lnodeR}}
			\node[lnodesmall] (1) at (5.6631, 4.1145) {1};
			\node[lnodesmall] (2) at (-1.2542, 2.8972) {2};
			\node[lnodesmall] (3) at (2.2314, -1.9866) {3};
			\node[lnodesmall] (4) at (2.1776, 8.9982) {4};
			\node[lnodesmall] (5) at (0.00000, 0.00000) {5};
			\node[lnodesmall] (6) at (4.0000, 0.00000) {6};
			\node[lnodesmall] (7) at (2.0000, 0.00000) {7};
				\draw[edge, col1](6)edge(7) (5)edge(7) (3)edge(7);
				\draw[edge, col2](1)edge(5) (1)edge(3) (2)edge(4);
				\draw[edge, col3](2)edge(6) (1)edge(4) (2)edge(3);
				\draw[edge](4)edge(7);
		\end{tikzpicture}
	\end{tabular}
	\caption{The graph $Q_1$ has 8 families of proper flexible labelings.
	Figures~\ref{fig:Q1_case_I_example} and \ref{fig:Q1_anim_IIplus} show two of them.
	Representatives for the remaining ones are depicted here.}
	\label{fig:Q1remaining}
\end{figure}

\begin{figure}[htb]
	\centering
	\begin{tabular}{cccccc}
		\begin{tikzpicture}[scale=0.8]
		    \node[lnodesmall] (1) at (-1.0, 0.00) {1};
		    \node[lnodesmall] (2) at (1.0, 0.00) {2};
		    \node[lnodesmall] (3) at (-0.50, 0.866025000000000) {3};
		    \node[lnodesmall] (4) at (0.50, 0.866025000000000) {4};
		    \node[lnodesmall] (5) at (-0.50, -0.866025000000000) {5};
		    \node[lnodesmall] (6) at (0.50, -0.866025000000000) {6};
		    \node[lnodesmall] (7) at (0, 0) {7};
		    \draw[redge] (1)edge(3) (2)edge(4) (2)edge(6) (4)edge(7) (5)edge(6) (5)edge(7) (6)edge(7)  ;
		    \draw[bedge] (1)edge(4) (1)edge(5) (2)edge(3) (3)edge(7)  ;
		\end{tikzpicture}
		&
		\begin{tikzpicture}[scale=0.8]
		    \node[lnodesmall] (1) at (-1.0, 0.00) {1};
		    \node[lnodesmall] (2) at (1.0, 0.00) {2};
		    \node[lnodesmall] (3) at (-0.50, 0.866025000000000) {3};
		    \node[lnodesmall] (4) at (0.50, 0.866025000000000) {4};
		    \node[lnodesmall] (5) at (-0.50, -0.866025000000000) {5};
		    \node[lnodesmall] (6) at (0.50, -0.866025000000000) {6};
		    \node[lnodesmall] (7) at (0, 0) {7};
		    \draw[redge] (1)edge(4) (2)edge(3) (2)edge(6) (3)edge(7) (5)edge(6) (5)edge(7) (6)edge(7)  ;
		    \draw[bedge] (1)edge(3) (1)edge(5) (2)edge(4) (4)edge(7)  ;
		\end{tikzpicture}
		&
		\begin{tikzpicture}[scale=0.8]
		    \node[lnodesmall] (1) at (-1.0, 0.00) {1};
		    \node[lnodesmall] (2) at (1.0, 0.00) {2};
		    \node[lnodesmall] (3) at (-0.50, 0.866025000000000) {3};
		    \node[lnodesmall] (4) at (0.50, 0.866025000000000) {4};
		    \node[lnodesmall] (5) at (-0.50, -0.866025000000000) {5};
		    \node[lnodesmall] (6) at (0.50, -0.866025000000000) {6};
		    \node[lnodesmall] (7) at (0, 0) {7};
		    \draw[redge] (1)edge(4) (1)edge(5) (2)edge(3) (4)edge(7) (5)edge(6) (5)edge(7) (6)edge(7)  ;
		    \draw[bedge] (1)edge(3) (2)edge(4) (2)edge(6) (3)edge(7)  ;
		\end{tikzpicture}
		&
		\begin{tikzpicture}[scale=0.8]
		    \node[lnodesmall] (1) at (-1.0, 0.00) {1};
		    \node[lnodesmall] (2) at (1.0, 0.00) {2};
		    \node[lnodesmall] (3) at (-0.50, 0.866025000000000) {3};
		    \node[lnodesmall] (4) at (0.50, 0.866025000000000) {4};
		    \node[lnodesmall] (5) at (-0.50, -0.866025000000000) {5};
		    \node[lnodesmall] (6) at (0.50, -0.866025000000000) {6};
		    \node[lnodesmall] (7) at (0, 0) {7};
		    \draw[redge] (1)edge(3) (1)edge(5) (2)edge(4) (3)edge(7) (5)edge(6) (5)edge(7) (6)edge(7)  ;
		    \draw[bedge] (1)edge(4) (2)edge(3) (2)edge(6) (4)edge(7)  ;
		\end{tikzpicture}
		&
		\begin{tikzpicture}[scale=0.8]
		    \node[lnodesmall] (1) at (-1.0, 0.00) {1};
		    \node[lnodesmall] (2) at (1.0, 0.00) {2};
		    \node[lnodesmall] (3) at (-0.50, 0.866025000000000) {3};
		    \node[lnodesmall] (4) at (0.50, 0.866025000000000) {4};
		    \node[lnodesmall] (5) at (-0.50, -0.866025000000000) {5};
		    \node[lnodesmall] (6) at (0.50, -0.866025000000000) {6};
		    \node[lnodesmall] (7) at (0, 0) {7};
		    \draw[redge] (1)edge(3) (1)edge(4) (2)edge(6) (5)edge(6) (5)edge(7) (6)edge(7)  ;
		    \draw[bedge] (1)edge(5) (2)edge(3) (2)edge(4) (3)edge(7) (4)edge(7)  ;
		\end{tikzpicture}
		&
		\begin{tikzpicture}[scale=0.8]
		    \node[lnodesmall] (1) at (-1.0, 0.00) {1};
		    \node[lnodesmall] (2) at (1.0, 0.00) {2};
		    \node[lnodesmall] (3) at (-0.50, 0.866025000000000) {3};
		    \node[lnodesmall] (4) at (0.50, 0.866025000000000) {4};
		    \node[lnodesmall] (5) at (-0.50, -0.866025000000000) {5};
		    \node[lnodesmall] (6) at (0.50, -0.866025000000000) {6};
		    \node[lnodesmall] (7) at (0, 0) {7};
		    \draw[redge] (1)edge(5) (2)edge(3) (2)edge(4) (5)edge(6) (5)edge(7) (6)edge(7)  ;
		    \draw[bedge] (1)edge(3) (1)edge(4) (2)edge(6) (3)edge(7) (4)edge(7)  ;
		\end{tikzpicture}
		\\
		$\epsilon_{13}$ & $\epsilon_{14}$ & $\epsilon_{23}$ & $\epsilon_{24}$ & $\gamma_1$ & $\gamma_2$ \\[5pt]
		\begin{tikzpicture}[scale=0.8]
		    \node[lnodesmall] (1) at (-1.0, 0.00) {1};
		    \node[lnodesmall] (2) at (1.0, 0.00) {2};
		    \node[lnodesmall] (3) at (-0.50, 0.866025000000000) {3};
		    \node[lnodesmall] (4) at (0.50, 0.866025000000000) {4};
		    \node[lnodesmall] (5) at (-0.50, -0.866025000000000) {5};
		    \node[lnodesmall] (6) at (0.50, -0.866025000000000) {6};
		    \node[lnodesmall] (7) at (0, 0) {7};
		    \draw[redge] (1)edge(3) (1)edge(4) (2)edge(3) (2)edge(4) (5)edge(6) (5)edge(7) (6)edge(7)  ;
		    \draw[bedge] (1)edge(5) (2)edge(6) (3)edge(7) (4)edge(7)  ;
		\end{tikzpicture}
		&
		\begin{tikzpicture}[scale=0.8]
		    \node[lnodesmall] (1) at (-1.0, 0.00) {1};
		    \node[lnodesmall] (2) at (1.0, 0.00) {2};
		    \node[lnodesmall] (3) at (-0.50, 0.866025000000000) {3};
		    \node[lnodesmall] (4) at (0.50, 0.866025000000000) {4};
		    \node[lnodesmall] (5) at (-0.50, -0.866025000000000) {5};
		    \node[lnodesmall] (6) at (0.50, -0.866025000000000) {6};
		    \node[lnodesmall] (7) at (0, 0) {7};
		    \draw[redge] (2)edge(3) (2)edge(4) (2)edge(6) (3)edge(7) (4)edge(7) (5)edge(6) (5)edge(7) (6)edge(7)  ;
		    \draw[bedge] (1)edge(3) (1)edge(4) (1)edge(5)  ;
		\end{tikzpicture}
		&
		\begin{tikzpicture}[scale=0.8]
		    \node[lnodesmall] (1) at (-1.0, 0.00) {1};
		    \node[lnodesmall] (2) at (1.0, 0.00) {2};
		    \node[lnodesmall] (3) at (-0.50, 0.866025000000000) {3};
		    \node[lnodesmall] (4) at (0.50, 0.866025000000000) {4};
		    \node[lnodesmall] (5) at (-0.50, -0.866025000000000) {5};
		    \node[lnodesmall] (6) at (0.50, -0.866025000000000) {6};
		    \node[lnodesmall] (7) at (0, 0) {7};
		    \draw[redge] (1)edge(3) (1)edge(4) (1)edge(5) (3)edge(7) (4)edge(7) (5)edge(6) (5)edge(7) (6)edge(7)  ;
		    \draw[bedge] (2)edge(3) (2)edge(4) (2)edge(6)  ;
		\end{tikzpicture}
		&
		\begin{tikzpicture}[scale=0.8]
		    \node[lnodesmall] (1) at (-1.0, 0.00) {1};
		    \node[lnodesmall] (2) at (1.0, 0.00) {2};
		    \node[lnodesmall] (3) at (-0.50, 0.866025000000000) {3};
		    \node[lnodesmall] (4) at (0.50, 0.866025000000000) {4};
		    \node[lnodesmall] (5) at (-0.50, -0.866025000000000) {5};
		    \node[lnodesmall] (6) at (0.50, -0.866025000000000) {6};
		    \node[lnodesmall] (7) at (0, 0) {7};
		    \draw[redge] (1)edge(4) (1)edge(5) (2)edge(4) (2)edge(6) (4)edge(7) (5)edge(6) (5)edge(7) (6)edge(7)  ;
		    \draw[bedge] (1)edge(3) (2)edge(3) (3)edge(7)  ;
		\end{tikzpicture}
		&
		\begin{tikzpicture}[scale=0.8]
		    \node[lnodesmall] (1) at (-1.0, 0.00) {1};
		    \node[lnodesmall] (2) at (1.0, 0.00) {2};
		    \node[lnodesmall] (3) at (-0.50, 0.866025000000000) {3};
		    \node[lnodesmall] (4) at (0.50, 0.866025000000000) {4};
		    \node[lnodesmall] (5) at (-0.50, -0.866025000000000) {5};
		    \node[lnodesmall] (6) at (0.50, -0.866025000000000) {6};
		    \node[lnodesmall] (7) at (0, 0) {7};
		    \draw[redge] (1)edge(3) (1)edge(5) (2)edge(3) (2)edge(6) (3)edge(7) (5)edge(6) (5)edge(7) (6)edge(7)  ;
		    \draw[bedge] (1)edge(4) (2)edge(4) (4)edge(7)  ;
		\end{tikzpicture}
		&
		\begin{tikzpicture}[scale=0.8]
		    \node[lnodesmall] (1) at (-1.0, 0.00) {1};
		    \node[lnodesmall] (2) at (1.0, 0.00) {2};
		    \node[lnodesmall] (3) at (-0.50, 0.866025000000000) {3};
		    \node[lnodesmall] (4) at (0.50, 0.866025000000000) {4};
		    \node[lnodesmall] (5) at (-0.50, -0.866025000000000) {5};
		    \node[lnodesmall] (6) at (0.50, -0.866025000000000) {6};
		    \node[lnodesmall] (7) at (0, 0) {7};
		    \draw[redge] (1)edge(5) (2)edge(6) (5)edge(6) (5)edge(7) (6)edge(7)  ;
		    \draw[bedge] (1)edge(3) (1)edge(4) (2)edge(3) (2)edge(4) (3)edge(7) (4)edge(7)  ;
		\end{tikzpicture}
		\\
		$\eta$ & $\psi_1$ & $\psi_2$ & $\phi_3$ & $\phi_4$ & $\zeta$
	\end{tabular}
	\caption{NAC-colorings of the graph $Q_1$, modulo conjugation.}
	\label{fig:Q1NACs}
\end{figure}
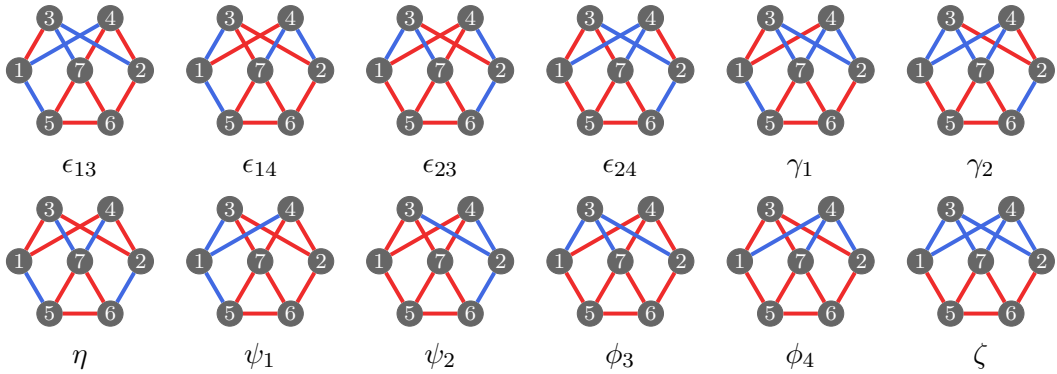

Table~\ref{tab:Q1_consistent_motion_types} summarizes the computation of consistent motion types,
where $H_1, \dots, H_7$ are the 4-cycles $(1, 3, 2, 4)$, $(1, 3, 7, 4)$, $(2, 3, 7, 4)$, 
$(1, 3, 7, 5)$, $(1, 4, 7, 5)$, $(2, 4, 7, 6)$, $(2, 3, 7, 6)$ respectively.
The column ``Motion family'' indicates to which family 
of proper flexible labelings of $Q_1$ a particular row belongs.
The next column gives information on the dimension of this family.
\begin{table}[ht]
	\centering
		\begin{tabular}{c*{7}{C@{\hspace{6pt}}}cc@{\hspace{5pt}}cc@{\hspace{5pt}}c}
			& \multicolumn{7}{l}{\textbf{Motions types}} & \textbf{Active} & \textbf{\#} & \textbf{Motion}
				& \multirow{2}{*}{\textbf{Dim.}} & \multirow{2}{*}{\textbf{}} \\
			& \multicolumn{7}{l}{\textbf{ of} $(H_1,\dots,H_7)$} & \textbf{NAC-colorings} 
					& \textbf{isom.} & \textbf{family} \\ 
			\midrule
			\newCaseQone\label{Q1:a} & \caseP{} & \caseG{} & \caseG{} & \caseP{} & \caseG{} & \caseP{} & \caseG{} &
				 $\{\epsilon_{13}, \epsilon_{24}, \eta\}$
				  & $2$ & I & 4\\
			\newCaseQone\label{Q1:b} & \caseP{} & \caseO{} & \caseA{} & \caseP{} & \caseO{} & \caseP{} & \caseE{} &
				 $\{\epsilon_{13}, \eta\}$
				  & $4$ & $\!\subset$ I, $\IVminus$, V, VI & 2 & Sec. \ref{subsec:Q1_typeI} \\
			\newCaseQone\label{Q1:c} & \caseP{} & \caseE{} & \caseE{} & \caseP{} & \caseA{} & \caseP{} & \caseA{} &
				 $\{\epsilon_{13}, \epsilon_{24}\}$
				  & $2$ & $\subset$ I, II, III & 2\\
			\midrule
			\newCaseQone\label{Q1:d} & \caseO{} & \caseG{} & \caseG{} & \caseG{} & \caseG{} & \caseG{} & \caseG{} &
				 $\{\epsilon_{ij}, \gamma_{1}, \gamma_{2}, \psi_{1}, \psi_{2}\}$
				 & $1$ & $\IIminus\cup \IIplus$  & 5 &\multirow{4}{*}{Sec. \ref{subsec:Q1_typeII}} \\
			\newCaseQone\label{Q1:e} & \caseP{} & \caseE{} & \caseE{} & \caseG{} & \caseG{} & \caseG{} & \caseG{} &
				 $\{\epsilon_{13}, \epsilon_{14}, \epsilon_{23}, \epsilon_{24}\}$
				 & $1$ & $\subset \IIminus, \IIplus$ & 4 \\
			\newCaseQone\label{Q1:f} & \caseO{} & \caseG{} & \caseG{} & \caseP{} & \caseG{} & \caseG{} & \caseA{} &
				 $\{\epsilon_{13}, \epsilon_{24}, \gamma_{1}, \psi_{2}\}$
				 & $4$ & $\subset\IIminus$ & 3 \\
			\newCaseQone\label{Q1:g} & \caseO{} & \caseG{} & \caseG{} & \caseE{} & \caseG{} & \caseG{} & \caseE{} &
				 $\{\epsilon_{13}, \epsilon_{23}, \gamma_{1}, \gamma_{2}\}$
				 & $2$ & $\subset\IIminus$, deg. & 2\\
			\midrule
			\newCaseQone\label{Q1:h} & \caseO{} & \caseG{} & \caseG{} & \caseG{} & \caseA{} & \caseG{} & \caseA{} &
				 $\{\epsilon_{13}, \epsilon_{24}, \psi_{1}, \psi_{2}, \zeta\}$
				 & $2$ & III & 3 & \cite{ClassificationQ1}\\
			\midrule
			\newCaseQone\label{Q1:i} & \caseG{} & \caseG{} & \caseA{} & \caseP{} & \caseG{} & \caseG{} & \caseG{} &
				 $\{\epsilon_{13}, \eta, \phi_{4}, \psi_{2}\}$
				 & $4$ & $\IVminus\cup \IVplus$ & 4 & \cite{ClassificationQ1}\\
			\midrule
			\newCaseQone\label{Q1:j} & \caseG{} & \caseG{} & \caseA{} & \caseE{} & \caseG{} & \caseP{} & \caseE{} &
				 $\{\epsilon_{13}, \eta, \gamma_{2}, \phi_{3}\}$
				 & $4$ & V & 3 & \cite{ClassificationQ1}\\
			\midrule
			\newCaseQone\label{Q1:k} & \caseP{} & \caseG{} & \caseG{} & \caseE{} & \caseG{} & \caseG{} & \caseE{} &
				 $\{\epsilon_{13}, \epsilon_{23}, \eta, \zeta\}$
				 & $2$ & VI & 3 & \cite{ClassificationQ1}\\
		\end{tabular}
	\caption{The cases of consistent motion types and active NAC-colorings of $Q_1$.}
	\label{tab:Q1_consistent_motion_types}
\end{table}

\subsection{Motion family I}
\label{subsec:Q1_typeI}
We prove that Case~\ref{Q1:a} gives a  4-dimensional family of proper flexible labelings.
These proper flexible labelings can be actually constructed by \cite[Lemma~4.4]{movableGraphs}.
The motion types in Case \ref{Q1:a} enforce some equal edge lengths:
$\textcolor{col1}{\lambda_{13} =  \lambda_{24} = \lambda_{57} =  \lambda_{67}}$,
$\textcolor{col4}{\lambda_{14} =  \lambda_{23}}$, 
$\textcolor{col3}{\lambda_{15} =  \lambda_{37}}$,
$\textcolor{col2}{\lambda_{26} =  \lambda_{47}}$
(see also Figure~\ref{fig:Q1_edge_lengths_ad}).
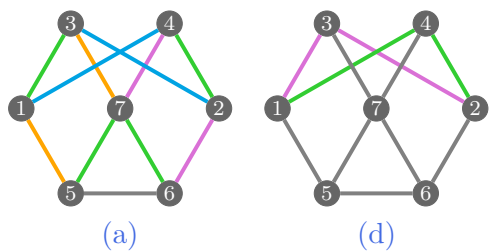
\begin{figure}[b]
	\centering
		\begin{tabular}{ccc}
			\begin{tikzpicture}[scale=1.3]
				\node[lnodesmall] (1) at (-1.0000, 0.00000) {1};
				\node[lnodesmall] (2) at (1.0000, 0.00000) {2};
				\node[lnodesmall] (3) at (-0.50000, 0.86600) {3};
				\node[lnodesmall] (4) at (0.50000, 0.86600) {4};
				\node[lnodesmall] (5) at (-0.50000, -0.86600) {5};
				\node[lnodesmall] (6) at (0.50000, -0.86600) {6};
				\node[lnodesmall] (7) at (0.00000, 0.00000) {7};
				\draw[edge,col1](6)edge(7) (5)edge(7) (1)edge(3) (2)edge(4);
				\draw[edge,col2](4)edge(7) (2)edge(6);
				\draw[edge,col3](1)edge(5) (3)edge(7);
				\draw[edge,col4](1)edge(4) (2)edge(3);
				\draw[edge](5)edge(6);
			\end{tikzpicture}
			&
			\begin{tikzpicture}[scale=1.3]
				\node[lnodesmall] (1) at (-1.0000, 0.00000) {1};
				\node[lnodesmall] (2) at (1.0000, 0.00000) {2};
				\node[lnodesmall] (3) at (-0.50000, 0.86600) {3};
				\node[lnodesmall] (4) at (0.50000, 0.86600) {4};
				\node[lnodesmall] (5) at (-0.50000, -0.86600) {5};
				\node[lnodesmall] (6) at (0.50000, -0.86600) {6};
				\node[lnodesmall] (7) at (0.00000, 0.00000) {7};
				\draw[edge, col2](1)edge(3) (2)edge(3);
				\draw[edge, col1](1)edge(4) (2)edge(4);
				\draw[edge](4)edge(7) (2)edge(6) (6)edge(7) (5)edge(6) (5)edge(7) (1)edge(5) (3)edge(7);
			\end{tikzpicture}
			\\
			\ref{Q1:a} & \ref{Q1:d}
		\end{tabular}
	\caption{Edge lengths enforced by motion types~\ref{Q1:a} and \ref{Q1:d}: same color indicates equality of edge lengths (except for gray).}
	\label{fig:Q1_edge_lengths_ad}
\end{figure}
Since the NAC-colorings $\epsilon_{13}, \epsilon_{24}$ and $\eta$ are singletons,
we can compare leading coefficients as described in Section~\ref{sec:comparingLC}
to obtain the following equations for~$\lambda$:
\begin{align}
	\label{eq:Q1_lcs_abc}
	(\lambda_{56} - 2 \lambda_{67})  (\lambda_{56} + 2 \lambda_{67})  
		(\lambda_{47} -  \lambda_{67})  (\lambda_{47} + \lambda_{67})&=0 & & (\epsilon_{13} \text{ active})\,, \notag \\
	(\lambda_{56} - 2 \lambda_{67})  (\lambda_{56} + 2 \lambda_{67})  
		(\lambda_{37} -  \lambda_{67})  (\lambda_{37} + \lambda_{67})&=0 & & (\epsilon_{24} \text{ active})\,, \\
	(\lambda_{56} - 2 \lambda_{67})  (\lambda_{56} + 2 \lambda_{67})  
		(\lambda_{23} -  \lambda_{67})  (\lambda_{23} + \lambda_{67})&=0 & & (\eta \text{ active})\,. \notag
\end{align}
The second and fourth factor cannot vanish due to positive edge lengths.
If the first factor does not vanish,
then $\lambda_{13} =  \lambda_{24} =\lambda_{57} =  \lambda_{67}= \lambda_{14} 
=  \lambda_{23}= \lambda_{15} =  \lambda_{37}= \lambda_{26} =  \lambda_{47}$,
which contradicts injective realizations
(e.\,g.\ the $K_{2,3}$ subgraph induced by $1, 2, 3, 4$ and~$7$ has no injective realization).
Hence, $\lambda_{56} = 2 \lambda_{67} = \lambda_{57} + \lambda_{67}$,
i.e., the triangle $(5,6,7)$ is degenerate.

If we fix the vertices $5$ and $6$ and
consider $\lambda_{67},\lambda_{23}, \lambda_{37}$ and $\lambda_{47}$ also as variables,
then the zero set of the system~\eqref{eq:mainSystemOfEquations}
has dimension $5$, as one can check by Gr\"obner basis computation.
There are only $4$ parameters, therefore if we fix $\lambda_{67},\lambda_{23}, \lambda_{37}$ and $\lambda_{47}$,
then there is a curve of solutions. Hence, the labeling is flexible whenever it is realizable.
If the parameters are general enough, for instance pairwise distinct,
then the labeling is proper flexible.
A realization compatible with such an instance is shown in Figure~\ref{fig:Q1_case_I_example}.
Cases \ref{Q1:b} and \ref{Q1:c} are special cases of Case \ref{Q1:a}, see~\cite{ClassificationQ1}.

\begin{figure}[ht]
	\centering
		\begin{tabular}{ccc}
			\begin{tikzpicture}[scale=0.9]
				\tikzset{lnodesmall/.style={lnodeR}}
				\node[lnodesmall] (1) at (-1700/1189, 2020/1189) {1};
				\node[lnodesmall] (2) at (68/41, 120/41) {2};
				\node[lnodesmall] (3) at (-511/1189, 2020/1189) {3};
				\node[lnodesmall] (4) at (27/41, 120/41) {4};
				\node[lnodesmall] (5) at (-1, 0) {5};
				\node[lnodesmall] (6) at (1, 0) {6};
				\node[lnodesmall] (7) at (0, 0) {7};
				\draw[edge,col1](6)edge(7) (5)edge(7) (1)edge(3) (2)edge(4);
				\draw[edge,col2](4)edge(7) (2)edge(6);
				\draw[edge,col3](1)edge(5) (3)edge(7);
				\draw[edge,col4](1)edge(4) (2)edge(3);
			\end{tikzpicture}
			&
			\begin{tikzpicture}[scale=2]
				\tikzset{lnodesmall/.style={lnodeR}}
				\node[lnodesmall] (1) at (-0.10443, 0.48897) {1};
				\node[lnodesmall] (2) at (1.6090, 1.3708) {2};
				\node[lnodesmall] (3) at (0.39557, 0.48897) {3};
				\node[lnodesmall] (4) at (1.1090, 1.3708) {4};
				\node[lnodesmall] (5) at (0.00000, 0.00000) {5};
				\node[lnodesmall] (6) at (1.0000, 0.00000) {6};
				\node[lnodesmall] (7) at (0.50000, 0.00000) {7};
				\draw[edge,col2](4)edge(7) (2)edge(6) (1)edge(4) (2)edge(3);
				\draw[edge,col1](6)edge(7) (5)edge(7) (1)edge(5) (1)edge(3) (3)edge(7) (2)edge(4);
			\end{tikzpicture}
			&
			\begin{tikzpicture}
				\tikzset{lnodesmall/.style={lnodeR}}
				\node[lnodesmall] (1) at (2.4268, 1.7637) {1};
				\node[lnodesmall] (2) at (3.2179, 2.7417) {2};
				\node[lnodesmall] (3) at (3.4268, 1.7637) {3};
				\node[lnodesmall] (4) at (2.2179, 2.7417) {4};
				\node[lnodesmall] (5) at (0.00000, 0.00000) {5};
				\node[lnodesmall] (6) at (2.0000, 0.00000) {6};
				\node[lnodesmall] (7) at (1.0000, 0.00000) {7};
				\draw[edge,col1](6)edge(7) (5)edge(7) (1)edge(4) (1)edge(3) (2)edge(3) (2)edge(4);
				\draw[edge,col2](4)edge(7) (2)edge(6) (1)edge(5) (3)edge(7);
			\end{tikzpicture}
			\\
			\ref{Q1:a} & \ref{Q1:b} & \ref{Q1:c}
		\end{tabular}
		\caption{Realizations of $Q_1$ compatible with proper flexible labelings of family~I.}
		\label{fig:Q1_case_I_example}
\end{figure}
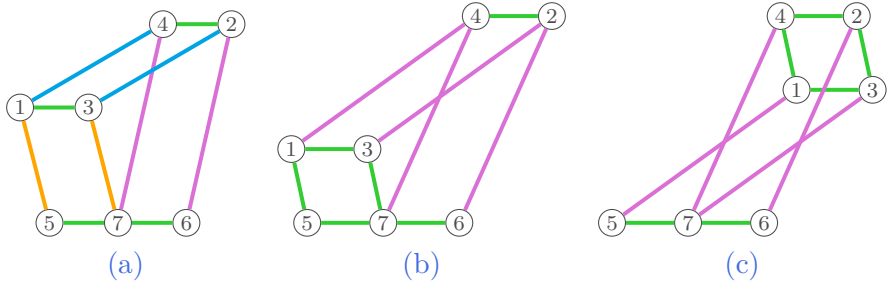

\subsection{Motion families \texorpdfstring{$\IIminus$ and $\IIplus$}{II- and II+}}
\label{subsec:Q1_typeII}
We focus now on Case~\ref{Q1:d}, since the Cases~\ref{Q1:e}--\ref{Q1:g} are its subcases~\cite{ClassificationQ1}.
If $\gamma_1$, resp.\ $\gamma_2$ is active,
then the 4-cycle $H_3=(2, 3, 7, 4)$, resp.\ $H_2=(1, 3, 7, 4)$
has orthogonal diagonals by Proposition~\ref{prop:orthogonalDiagonals}.
Together with the motion types of $H_1, H_2$ and $H_3$,
this implies that the vertices $1, 2$ and $7$ are collinear.
The motion types enforce
$\lambda_{13} =  \lambda_{23} \text{ and } \lambda_{14} =  \lambda_{24}$ (see Figure~\ref{fig:Q1_edge_lengths_ad}).
The triangle $(5, 6, 7)$ must be  degenerate,
otherwise Proposition~\ref{prop:degenerateTriangle} for~$\epsilon_{13}$ gives  $\lambda_{24}=\lambda_{47}$,
but three distinct collinear points $1, 2$ and $7$ cannot 
have the same distance to vertex $4$.
Allowing also negative edge lengths in the degenerate triangle,
all three cases are covered by the equation $\lambda_{56} = \lambda_{57} + \lambda_{67}$.

The active NAC-colorings $\epsilon_{13}, \epsilon_{14}, \epsilon_{23}, \epsilon_{24}, \gamma_{1}$ and $\gamma_{2}$
are singletons.
The zero set of the system of equations they provide together with the enforced edge lengths and the triangle equality $\lambda_{56} = \lambda_{57} + \lambda_{67}$
has two irreducible components $\IIminus$ and $\IIplus$ given by
\begin{align*}
- \lambda_{24}^{2} + \alpha\lambda_{15} \lambda_{26} + \lambda_{47}^{2} + \lambda_{57} \lambda_{67} =0\,, \quad
\lambda_{26} \lambda_{57} + \alpha \lambda_{15} \lambda_{67} &=0\,, \\
\lambda_{13} =  \lambda_{23} \,, \quad
\lambda_{14} =  \lambda_{24} \,, \quad
\lambda_{24}^{2} + \lambda_{37}^{2} = \lambda_{23}^{2} + \lambda_{47}^{2} \,, \quad
\lambda_{57} + \lambda_{67} &= \lambda_{56} \,,
\end{align*}
where $\alpha\in\{-1, 1\}$ (see~\cite{LegerskySupportingMaterial}).
The dimension of both varieties is 5.
If we construct the system of equations for vertex coordinates,
taking the $\lambda_{ij}$ as variables, the dimension is~6.
Therefore, a generic fiber of the projection from the whole zero set
to $\IIminus$, resp.~$\IIplus$, has positive dimension.
Hence, a generic $\lambda$ in $\IIminus\cup\IIplus$ that is realizable is flexible.
An example of a proper flexible labeling in $\IIplus$ is
$\lambda_{13} =\lambda_{23}=14$,
$\lambda_{15} =9 $,
$\lambda_{26} =12 $,
$\lambda_{37} =10 $,
$\lambda_{47} =5 $,
$\lambda_{14} =\lambda_{24}=11 $,
$\lambda_{56} =1$,
$\lambda_{57} =-3 $,
$\lambda_{67} =4$.
A parametrization of the motion can be found in~\cite{ClassificationQ1},
as well as an example of a proper flexible labeling in $\IIminus$.
Figure~\ref{fig:Q1_anim_IIplus} illustrates $\IIplus$.

\begin{figure}[htb]
	\centering
		\begin{tikzpicture}[scale=0.23]
			\tikzset{lnodesmall/.style={lnodeR}}
			\node[lnodesmall] (1) at (6.3640, 6.3640) {1};
			\node[lnodesmall] (2) at (-10.013, -4.7661) {2};
			\node[lnodesmall] (3) at (-7.3884, 8.9857) {3};
			\node[lnodesmall] (4) at (-4.5190, 4.7637) {4};
			\node[lnodesmall] (5) at (0.00000, 0.00000) {5};
			\node[lnodesmall] (6) at (1.0000, 0.00000) {6};
			\node[lnodesmall] (7) at (-3.0000, 0.00000) {7};
				\draw[edge, col2](1)edge(3) (2)edge(3);
				\draw[edge, col1](1)edge(4) (2)edge(4);
				\draw[edge](4)edge(7) (2)edge(6) (5)edge(6) (5)edge(7) (1)edge(5) (3)edge(7);
			\node[lnodesmall] (5) at (5) {5};
			\begin{scope}[xshift=22cm]
				\node[lnodesmall] (1) at (-2.3294, 8.6933) {1};
				\node[lnodesmall] (2) at (-3.8468, -10.978) {2};
				\node[lnodesmall] (3) at (-12.993, -0.37806) {3};
				\node[lnodesmall] (4) at (-7.9407, -0.76780) {4};
				\node[lnodesmall] (5) at (0.00000, 0.00000) {5};
				\node[lnodesmall] (6) at (1.0000, 0.00000) {6};
				\node[lnodesmall] (7) at (-3.0000, 0.00000) {7};
				\draw[edge, col2](1)edge(3) (2)edge(3);
				\draw[edge, col1](1)edge(4) (2)edge(4);
				\draw[edge](4)edge(7) (2)edge(6) (5)edge(6) (5)edge(7) (1)edge(5) (3)edge(7);
			\node[lnodesmall] (5) at (5) {5};
			\end{scope}
			\begin{scope}[xshift=36cm]
				\node[lnodesmall] (1) at (-6.3640, 6.3640) {1};
				\node[lnodesmall] (2) at (3.2324, -11.791) {2};
				\node[lnodesmall] (3) at (-9.9800, -7.1610) {3};
				\node[lnodesmall] (4) at (-5.0555, -4.5579) {4};
				\node[lnodesmall] (5) at (0.00000, 0.00000) {5};
				\node[lnodesmall] (6) at (1.0000, 0.00000) {6};
				\node[lnodesmall] (7) at (-3.0000, 0.00000) {7};
				\draw[edge, col2](1)edge(3) (2)edge(3);
				\draw[edge, col1](1)edge(4) (2)edge(4);
				\draw[edge](4)edge(7) (2)edge(6) (5)edge(6) (5)edge(7) (1)edge(5) (3)edge(7);
			\node[lnodesmall] (5) at (5) {5};
			\end{scope}
		\end{tikzpicture}
	\caption{Example of a motion family $\IIplus$.}
	\label{fig:Q1_anim_IIplus}
\end{figure}
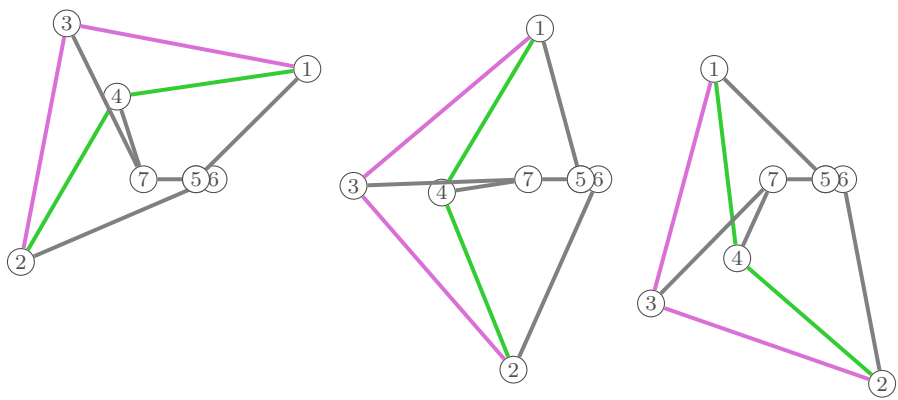

\section{Conclusion}
Possible cases of motions of a graph with algebraic constraints on edge lengths can be obtained
by the methods presented in this paper. 
They do, however, depend on computations and therefore on available computing power,
especially the elimination via Gr\"obner basis used in Section~\ref{sec:comparingLC}.
A full classification of all proper flexible labelings requires examining the cases one by one,
possibly using a computer algebra system.

\appendix 

\section{Active NAC-colorings of quadrilaterals}
\label{sec:NACsQuadrilateral}

Let $\lambda(v_1v_2)=\lambda_{12}$, $\lambda(v_2v_3)=\lambda_{23}$, $\lambda(v_3v_4)=\lambda_{34}$
and $\lambda(v_1v_4)=\lambda_{14}$
and assume that~$v_1v_2$ is the fixed edge.
For simplicity, we write $W_{i,j}$ and $Z_{i,j}$ instead of $W_{v_i,v_j}$ and $Z_{v_i,v_j}$.
 We determine the valuations $\nu$ such that $\nu(W_{1,2})=0$ and $\nu(W_{e})>0$ is positive for some other edge $e$.
 Namely, $\nu$ yields an active NAC-coloring by taking the threshold $\alpha=0$.
 Hence, the fixed edge $v_1v_2$ is always \blue{}.
 The conjugated NAC-colorings are automatically active \cite[Lemma~2.13]{movableGraphs}.
 The system of the equations \eqref{eq:equationCycles} has the form
 \begin{align*}
 	W_{1,2}= Z_{1,2}=\lambda_{12}\,, \quad  W_{2,3} Z_{2,3} &= \lambda_{23}^{2}\,, &
 	W_{3,4} Z_{3,4} = \lambda_{34}^{2}\,, \quad  W_{4,1} Z_{4,1} &= \lambda_{14}^{2}\,, \\
 	W_{1,2}+W_{2,3} + W_{3,4} + W_{4,1}  &= 0 \,, & Z_{1,2}+Z_{2,3} + Z_{3,4} + Z_{4,1} &= 0 \,.
 \end{align*}
 We express $W_{3,4}$ as a Laurent series in $s=W_{4,1}$, unless stated differently.
 A valuation is then given by the order.
 Hence, we obtain active NAC-colorings such that $v_1v_4$ is \red{}.
 By substituting
 \begin{align*}
 	W_{1,2} &= \lambda_{12}\,, & Z_{1,2} &= \lambda_{12}\,,  & W_{2,3} &= -s - W_{1,2} - W_{3,4}\,, \\
 	Z_{3,4} &= \frac{\lambda_{34}^{2}}{W_{3,4}}\,, & Z_{4,1} &= \frac{\lambda_{14}^{2}}{s}\,,  
 	   & Z_{2,3} &= -\frac{\lambda_{14}^{2}}{s} - Z_{1,2} - Z_{3,4}
 \end{align*}		 
 into $W_{2,3} Z_{2,3} = \lambda_{23}^{2}$, we obtain the following equation:
 \begin{equation}
 	\label{eq:quadrilateral}
 	( \lambda_{14}^{2} + \lambda_{12} s ) W_{3,4} ^2+( \lambda_{12} \lambda_{14}^{2} 
 	+ \lambda_{12}^{2} s - \lambda_{23}^{2} s + \lambda_{34}^{2} s + \lambda_{14}^{2} s 
 	+ \lambda_{12} s^{2} ) W_{3,4} + \lambda_{12} \lambda_{34}^{2} s + \lambda_{34}^{2} s^{2}=0
 \end{equation}
 We solve the equation for $W_{3,4}$ for various cases:
 
 \subsection*{Rhombus}
 If $\lambda_{12}=\lambda_{23}=\lambda_{34}=\lambda_{14}$, then \eqref{eq:quadrilateral} simplifies to 
 $(W_{3,4} + \lambda_{12}) (W_{3,4} + s) (\lambda_{12} + s) \lambda_{12} = 0$.
 If $-\lambda_{12}=s=W_{4,1}=-x_4-\ci y_4$,
 then $x_4=\lambda_{12}$ and $y_4=0$, i.e., the rhombus is degenerate so that vertices $v_2$ and $v_4$ coincide.
 Since $s=W_{4,1}$ is not transcendental in this case, we choose $W_{2,3}$ to have a positive valuation.
 Hence, the type of the NAC-coloring is L~(\colL{}).
 If $W_{3,4}=-s=-W_{4,1}$, then the rhombus degenerates so that $v_1$ and $v_3$ coincide and the type of the NAC-coloring is R (\colR{}).
 Finally, if $W_{3,4} =-\lambda_{12}$, then we get O (\colO{}).
 Note that vanishing any two factors of the equation simultaneously contradicts a motion.
 
 \subsection*{Parallelogram and Antiparallelogram}
 If $\lambda_{12}=\lambda_{34}$ and $\lambda_{23}=\lambda_{14}$, then \eqref{eq:quadrilateral} simplifies to
 $(W_{3,4} \lambda_{23}^{2} + W_{3,4} \lambda_{12} s + \lambda_{12}^{2} s + \lambda_{12} s^{2}) (W_{3,4} + \lambda_{12})=0$.
 If $W_{3,4}=- \lambda_{12}$, then the type of the active NAC-coloring of the parallel motion of the parallelogram is O (\colO{}).
 Otherwise, we have $W_{3,4}=-\frac{\lambda_{12}^{2} s + \lambda_{12} s^{2}}{\lambda_{23}^{2} + \lambda_{12} s}$;
 the parallelogram moves along the antiparallel irreducible component and by symmetry we have R (\colR{}) and~L~(\colL{}).
 
 \subsection*{Deltoid}
 Assume that the quadrilateral is an even deltoid, i.e., $\lambda_{12}=\lambda_{14}$ and $\lambda_{23}=\lambda_{34}$.
 The equation~\eqref{eq:quadrilateral} has the form
 $(W_{3,4}^{2} \lambda_{12} + W_{3,4} \lambda_{12}^{2} + W_{3,4} \lambda_{12} s + \lambda_{23}^{2} s) (\lambda_{12} + s)=0$.
 If $s=-\lambda_{12}$, then the deltoid degenerates,
 $s$ is not transcendental and by choosing  $W_{2,3}$ to have a positive valuation,
 we get a NAC-coloring of type L (\colL{}).
 The solutions corresponding to the non-degenerate motion are 
 \begin{equation*}
 W_{3,4} = -\frac{\lambda_{12}^{2} + \lambda_{12} u \pm \sqrt{\lambda_{12}^{4} + \lambda_{12}^{2} u^{2} + 2 \, {\left(\lambda_{12}^{3} - 2 \, \lambda_{12} \lambda_{23}^{2}\right)} u}}{2 \, \lambda_{12}}\,.
 \end{equation*}
 Using $\sqrt{1+t}=1 + \frac{t}{2}-\frac{t^2}{8}+O(t^3)$, they can be expressed as Laurent series:
 
 \begin{align*}
 	W_{3,4}=-\frac{\lambda_{23}^{2}}{\lambda_{12}^{2}}s + O(s^{2}) \quad
 	 \text{and} \quad W_{3,4}=-\lambda_{12} + \frac{\lambda_{23}^{2} - \lambda_{12}^{2}}{\lambda_{12}^{2}}s + O(s^{2})\,.
 \end{align*}
 This gives active NAC-colorings of type R (\colR{}) and O (\colO{}).
 The types of active NAC-colorings of an odd deltoid are L (\colL{}) and O (\colO{}).
 
 \subsection*{General case}
 If the lengths are general, then \eqref{eq:quadrilateral} has the following solutions:
 \begin{equation*}
 	W_{3,4}=-\frac{\lambda_{34}^{2}}{\lambda_{14}^{2}}s + O(s^{2}) \quad \text{and} \quad W_{3,4}=-\lambda_{12} + \left(\frac{\lambda_{23}^{2} - \lambda_{14}^{2}}{\lambda_{14}^{2}}\right)s + O(s^{2})\,.
 \end{equation*}
 Since we can also choose $W_{2,3}$ to have a positive valuation,
 we get altogether NAC-colorings of all three types L (\colL), O (\colO{)} and R (\colR{}).

\end{document}